\documentclass[11pt,leqno,oneside]{amsart}
\usepackage[utf8]{inputenc}
\usepackage[T1]{fontenc}
\usepackage{amsmath,amsthm}
\usepackage{amssymb}
\usepackage[english]{babel}

\usepackage[left=2cm,top=2.5cm,bottom=2.5cm,right=2cm]{geometry}
\usepackage[colorlinks,citecolor=blue]{hyperref}
\everymath{\displaystyle}

\usepackage{physics}
\usepackage{amsmath}
\usepackage{tikz}
\usepackage{mathdots}
\usepackage{yhmath}
\usepackage{cancel}
\usepackage{color}
\usepackage{siunitx}
\usepackage{array}
\usepackage{multirow}
\usepackage{amssymb}
\usepackage{gensymb}
\usepackage{tabularx}
\usepackage{extarrows}
\usepackage{booktabs}
\usetikzlibrary{fadings}
\usetikzlibrary{patterns}
\usetikzlibrary{shadows.blur}
\usetikzlibrary{shapes}

\newcommand{\elld}{\ell_{v,d}}
\newcommand{\ve}{\varepsilon}

\newtheorem{prop}{Proposition}
\newtheorem{lemma}{Lemma}
\newtheorem{thm}{Theorem}

\theoremstyle{definition}
\newtheorem{definition}{Definition}
\newtheorem{rem}{Remark}

\author{Stefano Baranzini}
\email{stefano.baranzini@unito.it}
\thanks{The author is partially supported by the 2025 INDAM-GNAMPA project "Non integrabilità e complessità in Meccanica Celeste" and by the 2022 PRIN project "Stability in Hamiltonian Mechanics and Beyond"}
\title{Non-integrability of n-centre billiards}
\address{University of Turin, Department of Mathematics, via Carlo Alberto 10, Turin, Italy}

\date{}


\makeatletter
\@namedef{subjclassname@2020}{\textup{2020} Mathematics Subject Classification}
\makeatother


\begin{document}
	\begin{abstract}
		We investigate a class of mechanical billiards, where a particle moves in a planar region under the influence of an $n$–centre potential and reflects elastically on a straight wall. Motivated by Boltzmann’s original billiard model we explore when such systems fail to be integrable. We show that for any positive energy, if there are at least two centres and the wall is placed sufficiently far away, the billiard dynamics is necessarily non-integrable and exhibits chaotic behaviour. In the classical Newtonian two-centre problem, we identify a simple geometric condition on the wall that likewise guarantees chaos, even though the free two-centre system is integrable. 
	\end{abstract}
	\maketitle
	
	\section{Introduction}
	Classical billiards describe the motion of a free particle confined in a domain $\Omega\subset \mathbb{R}^2$. When the particle hits the boundary, it gets reflected elastically, preserving the tangential component of its velocity and switching sign to the normal one. Billiard dynamics is a thriving area or research and ``a playground for mathematicians''\footnote{See \cite{katok_playground} for a pleasant survey on the topic.} where many qualitatively different behaviours can be observed. Despite its simple definition, it displays rich and diverse features depending on the shape of $\Omega$, many of which are not yet completely understood. 	
	For instance, it is still not known whether there are other integrable billiard tables other than circles and ellipses. Only local and partial solutions to this conjecture are known; see for instance \cite{birkhoff_conj_ADK,birkhoff_conj_SK,birkhoff_conj_BM,birkhoff_conj_G}.
	
	In the literature there are various generalizations of classical billiards. For instance, one can look at a billiard as a map on the space of oriented secant lines of $\Omega$. The image of a line is its orthogonal reflection about the tangent at the intersection point with $\partial \Omega$. One can change this rule by replacing the normal with another transverse vector field along the boundary. These are the \emph{projective billiards} introduced in \cite{TABACHNIKOV_1997}. Taking a dual perspective, one can define a map on the space of tangent lines to $\Omega$. This is known as \emph{outer billiard} and was introduced in \cite{moser_stable_random}.  One can leave the realm of Euclidean geometry and replace $\Omega$ with an open set of a Riemannian manifold, using the geodesic flow instead of straight lines to move from boundary point to boundary point. The list is much longer and constantly growing.
	
	In this paper, we consider a closely related variant of the latter system: planar \emph{mechanical billiards}. We still consider motion in a planar domain, but under the influence of a potential field. Upon impact with the boundary, the elastic collision condition holds. 	
	Many examples of integrable mechanical billiard systems have been found. For instance, any central force with a (centred) disk gives rise to an integrable mechanical billiard. More curiously, the harmonic oscillator potential $V(r)= kr^2$  and certain combinations of centred conics give rise to integrable mechanical billiards, see \cite{Pustovoitov_ellipses,Pustovoitov_conics}. Also the Kepler potential $V(r) = \frac{m}{r}$ combined with confocal conics is an integrable system as shown in \cite{kozlov1995,Kobtsev_elliptic}. Most notably, confocal ellipses are the only centrally symmetric, strictly convex, analytic domains containing the origin for which the Kepler billiard is integrable (see \cite{birkhoff_kepler, Barutello_2023}).   	
	 We refer to \cite{zhao_takeuchi_mechanical_billiards} for many other examples and an interesting discussion on how these systems are related through conformal changes of coordinates.

	In this paper, we consider the following mechanical billiards. Let us fix some points $c_1,\dots,c_n \in \mathbb{R}^2$, not necessarily distinct, which we call the \emph{centres}.	
    Let $\ell \subseteq \mathbb{R}^2\setminus \{c_1,\dots, c_n\}$ be a line and assume that $c_1,\dots c_n$ lie in a common half plane $\Pi_\ell$ bound by $\ell$. We consider a particle moving in $\Pi_\ell$, subject to the force exerted by the potential 
	\begin{equation}
		\label{eq:potential}
	V(x) = \sum_{i=1}^{n}\frac{m_i}{\vert x-c_i\vert^{\alpha_i}},
	\end{equation}
	where $\alpha_i\ge1$ and $m_i>0$. Upon impact with $\ell$, the velocity is reflected according to the classical law of elastic reflection. We will regard this system in two ways: as a three dimensional \emph{continuous time} one, following the trajectory of the particle in $\Pi_\ell$ or as a \emph{discrete time} one, iterating a symplectic map, the \emph{billiard map}, which keeps track only of consecutive bouncing points on the wall $\ell$ and their momenta.  Through out the paper we will restrict ourselves to positive energy levels, $h > 0$. 
	
	This work was initially motivated by the following system, which we call \emph{Boltzmann billiard}, considered by L. Boltzmann in \cite{boltzmann1868losung} as a candidate to illustrate his ergodic hypothesis. We have a particle, moving in a half-plane under the force field generated by 
	\[
	V_B(x) = \frac{1}{\vert x \vert}+\frac{\beta}{\vert x\vert^2}
	\]
	for some $\beta>0$. Upon impact with the boundary of the configuration space, the particle undergoes an elastic reflection. Boltzmann claimed that the corresponding billiard is ergodic. 	
	This claim has been disproved in \cite{felder2021poncelet} for small values of $\beta$ and negative energy $h<0$. The case $\beta=0$, also called \emph{Kepler billiard}, is known to be integrable for all values of the energy  (see \cite{gallavotti,zhao2022projective} and \cite{zhao_takeuchi_mechanical_billiards}) in the sense that there exists a non trivial function $I$ which is constant along orbits. In \cite{felder2021poncelet} it is shown that generic level sets of $I$ are foliated by quasi-periodic trajectories and so, as an  application of Moser's twist Theorem, the Boltzmann billiard is quasi-integrable as long as $\beta$ is sufficiently small. In this regime, the system cannot be ergodic.
	
 	We are thus prompt to consider other singular perturbations of Kepler potential and wonder: \emph{How special is Kepler billiard?}  \emph{Are there other $n-$centre type systems which are integrable?} It is classical that the Newtonian\footnote{i.e. with $\alpha_i=1$ for all $i$.}  $2-$centre problem is an integrable system. \emph{Is the corresponding billiard integrable?} This appears to be the case for certain combinations of confocal conics and for a straight line through the midpoint of the segment joining the centres, orthogonal to it, as shown in \cite[Section 5]{zhao_takeuchi_mechanical_billiards} and \cite{Takeuchi_2024}.
	
	On the other hand, for $n \ge 3$, the $n-$centre problem is non-integrable and has positive topological entropy, as shown in \cite{Bolotin_2017, BolKoz2017,baranzini2024chaotic}. In \cite{BolKoz2017, baranzini2024chaotic}, the authors proved  the existence of an invariant \emph{compact} chaotic subset. However, being compact, this set does not directly interact with the line $\ell$ whenever the latter sits sufficiently far from the centres (or the energy is high). Thus, it is natural to wonder, is the \emph{billiard map} integrable in this case? 
	
	The problem, at least in the high energy  setting, seems closely related to the \emph{irregular scattering} investigated in \cite{KleKna1992,Kna1987}. Stable and unstable manifolds of periodic trajectories extend to spatial infinity and the line $\ell$ acts as a focusing lens, enabling the particle to switch between trajectories shadowing the stable/unstable manifold of different periodic trajectories. 
	
	Therefore, it is natural to expect that the billiard map is not integrable whenever $n\ge 3$. Surprisingly though, we have non-integrability also in the case $n=2$ and, under some further mild additional assumptions on the potential\footnote{which are satisfied, for instance, by \emph{Boltzmann} billiards}, in the case with just one centre.
	
	\subsection{Main results}
	
	To state the two main Theorems of this work we need to introduce some notation. We consider a particle $x$ moving in $\mathbb{R}^2\setminus \{c_1,\dots,c_n\}$ and satisfying the following non-linear equation
	\begin{equation}
			\label{eq:motion_equation}
			\ddot{x}(t) =-\sum_{i=1 }^n \alpha_im_i\frac{x(t)-c_i}{\vert x(t)-c_i \vert^{\alpha_i+2}},
	\end{equation}
for $\alpha_i\ge 1$, $m_i>0$. Note that the $c_i$ are not necessarily distinct. Next, we introduce the reflection wall. 
\begin{definition}
	\label{def:mirror}
	Let $v,w\in \mathbb{R}^2$, $\vert v \vert = \vert w \vert =1$, $\langle v,w\rangle=0$. Let us assume that $\{v,w\}$ is positively oriented and let $d>0$. We consider the line
	\begin{equation}
		\label{eq:line}
		\ell_{v,d} =\langle v \rangle - d w .
	\end{equation}
    We say that $\elld$ is a reflection wall if $d$ is big enough so that $\elld$ 
	separates the plane in two half-planes, one of which contains all the $c_i$. We denote this half-plane by $\Pi_{\elld}$, i.e.
	\[
	\Pi_{\elld}  = \{x \in \mathbb{R}^2: \langle w,x\rangle \ge - d\}\supset \{c_1,\dots,c_n\}.
	\]
\end{definition}

The Hamiltonian function associated with \eqref{eq:motion_equation} is
\[
	H(u,x) = \frac{1}{2}\vert u \vert^2- V(x) = \frac{1}{2}\vert u \vert^2- \sum_{i=1}^{n}\frac{m_i}{\vert x-c_i\vert^{\alpha_i}}.
\]
For $h>0$, let us consider the set
\[
\Sigma_{h,\elld} = \{(u,x)\in \mathbb{R}^2: x \in \elld, \langle u, w \rangle > 0, H(u,x) =h\}.
\] 
Let $\Phi$ be the (partially defined) billiard map $\Phi: \Sigma_{h,\elld} \to \Sigma_{h,\elld}$ sending a point $(x,u)$ to the point of first return to $\Sigma_{h,\elld}$, with velocity reflected about $v$.
This leads to the following result.	
\begin{thm}
	\label{thm:main_perturbative}
Let $h>0$ and assume that $\#\{c_1,\dots,c_n\}\ge 2$. For every $v$ there exists $d_0(v,h)\footnote{If $\alpha_i \in [1,2)$ for all $i$, $d_0$ can be actually taken to be independent of the energy. Compare with Remark \ref{remark:ball_half_plane}. }>0$ such that, for every $d>d_0(v,h)$, the corresponding billiard map $\Phi:\Sigma_{h,\elld}\to \Sigma_{h,\elld}$ is not analytically integrable. Moreover, there exists an invariant compact subset of $\Sigma_{h,\elld}$ on which $\Phi$ is semi-conjugated to a Bernoulli shift.
\end{thm}

\begin{rem}
		Theorem \ref{thm:main_perturbative} says that, for any positive energy $h>0$, the billiard map is non-integrable provided that the line $\elld$ lies sufficiently far from the centres. The proof relies on a convexity property of solutions of \eqref{eq:motion_equation} known as \emph{Lagrange-Jacobi} inequality, compare with  \eqref{eq:lagrange_jacobi}. Of course, one could also phrase the statement fixing a line $\elld$ and increasing the energy of the system. Indeed, for any line $\elld$ there is an energy $h_0(v,d)$ such that for all $h\ge h_0(v,d)$ the Lagrange-Jacobi inequality holds and the billiard map is non-integrable.
	\end{rem}

\begin{rem}
	A closer look at the proof of Theorem \ref{thm:main_perturbative} shows that one can adapt the argument to the case in which the $c_i$ are all equal provided that at least one of the $\alpha_i>\frac{4}{3}$. Indeed, this condition (see \cite[Lemma 5.2 and Proposition 5.1]{BolKoz2017}) ensures that we can find enough collision-less non-homotopic length minimizers connecting points of $\elld$. The construction of an invariant set as in the statement of Theorem \ref{thm:main_perturbative} given in Section \ref{section:perturbative_version} carries over without any substantial change. Moreover, when the exponents $\alpha_i \in [1,2]$, inequality \eqref{eq:lagrange_jacobi} holds for all positive energies.
\end{rem}
   Our second main result focuses on the Newtonian $2-$centre problem. Let $e_1 = (1,0) \in \mathbb
   {R}^2$. Up to translation, rotation and scaling, it is not restrictive to consider the potential
   \[
    V(x) = \frac{m_1}{\vert x-e_1\vert}+ \frac{m_2}{\vert x+e_1\vert},\quad m_i>0.
   \]
    For any choice of line $\elld$ not crossing the segment $[-e_1,e_1]$, there exists two unique solutions of the $2-$centre problem, which we call $s_{\pm \infty}$, satisfying the following properties
    \begin{itemize}
    	\item $s_{\pm \infty} : [0,+\infty)\to \mathbb{R}\setminus[-e_1,e_1]$,
    	\item $s_{\pm \infty}$ are asymptotic to $[-e_1,e_1]$ and winding in opposite directions,
    	\item $s_{\pm\infty}(0)\in \elld$ and $\langle \dot s_{\pm\infty}(0),v\rangle=0$.
    \end{itemize} 
     The curves $s_{\pm\infty}$ belong to the stable/unstable manifold of the collision solution supported in $[-e_1,e_1]$. A precise description of these solutions is given in Sections \ref{subsec:curve_infinite_turns}.
     We now formulate an admissibility notion for a line $\elld$, in term of $s_{\pm \infty}$.
     \begin{definition}
     	\label{def:admissibility}
     	A line $\elld$ is an admissible reflection wall if the image of  $(0,\infty)$ through the $s_{\pm \infty}$ defined above is contained in the interior of $\Pi_{\elld}$. 
     \end{definition}
     \begin{figure}
     	\label{fig:admissible_lines}
     	\centering
        \caption{A solution $s_{\infty}$ and some admissible line. The mass parameter are $m_1=\frac32$,$m_2=\frac12$ and the energy is $h = \frac12$.}
     	\includegraphics[width=0.8\textwidth]{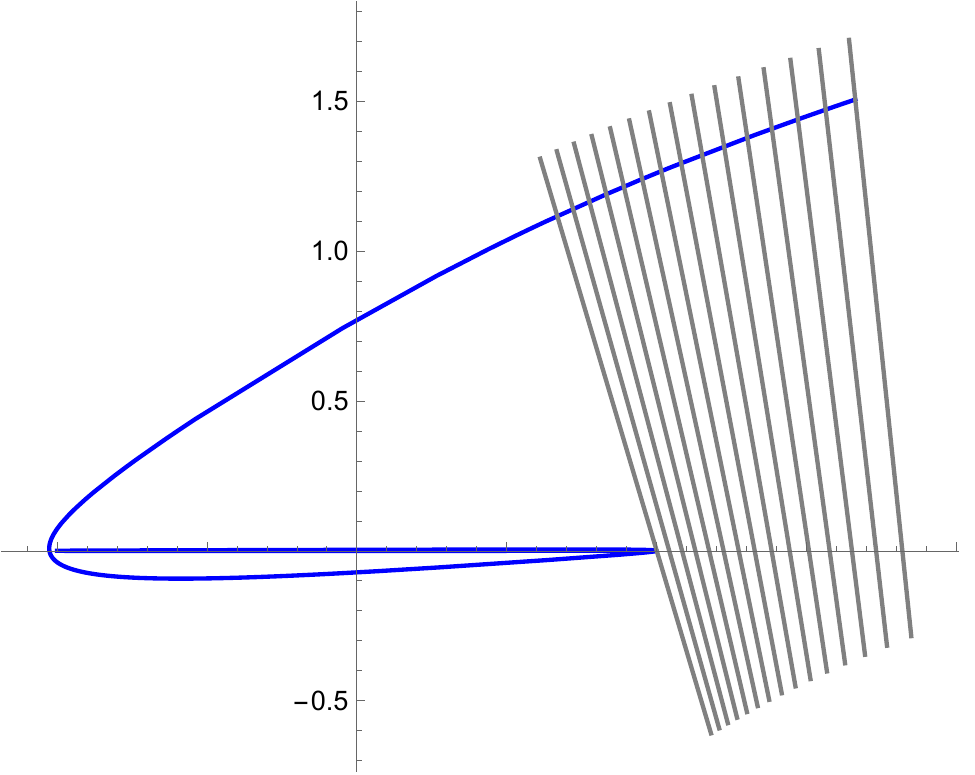}     
     \end{figure}
     Let us further remark that the solutions $s_{\pm \infty}$ are essentially explicit, compare with Appendix \ref{appendix:2_centres} and equations \eqref{eq:parametrization_xi}-\eqref{eq:parametrization_eta}.
    We have the following result.
    \begin{thm}
    	\label{thm:main_2_centres}
    	Let $\elld$ be an admissible reflection wall. The corresponding billiard map $\Phi:\Sigma_{h,\elld}\to \Sigma_{h,\elld}$ is not analytically integrable. Moreover, there exists a compact invariant subset of $\Sigma_{h,\elld}$ on which $\Phi$ is semi-conjugated to a sub-shift of finite type.
    \end{thm}
    \begin{rem}
    	Heuristically, Theorem \ref{thm:main_2_centres} is sharper than Theorem \ref{thm:main_perturbative}. Indeed, with the normalization $c_i = (-1)^i(1,0)$, the convexity assumption needed in the proof of Theorem \ref{thm:main_perturbative} holds outside a ball of radius $2$. On the contrary, the $s_{\pm \infty}$ converge exponentially fast to the collision-ejection solution (see Figure \ref{fig:admissible_lines}) and the invariant subset in Theorem \ref{thm:main_2_centres} is localized around them. The proof of Theorem \ref{thm:main_2_centres}, with due modifications, works also in the case $\#\{c_1,\dots,c_n\}>2$ or in the non purely Newtonian $2-$centre cases. However, since these systems are non-integrable (or expected to be non-integrable) this does not really constitute a quantitative improvement of Theorem \ref{thm:main_perturbative}.
    \end{rem}

    The structure of the paper is the following. In Section \ref{section:minimizers_JM_length} we collect some background material needed in the proofs. Section \ref{section:perturbative_version} is devoted to the construction of the invariant set of Theorem \ref{thm:main_perturbative} and Section \ref{section:2_centres} to the one of Theorem \ref{thm:main_2_centres}. In Section \ref{sec:symbolic} we set up explicitly the conjugation with a shift system. In the Appendix \ref{appendix:2_centres} we give some further details about the integrability of the $2-$centres problem and the trajectories $s_{\pm\infty}$.

\section{Variational set-up and the JM metric}
   \label{section:minimizers_JM_length}
\subsection{The Jacobi-Maupertuis metric}
   Let us recall the so-called \emph{Jacobi-Maupertuis} principle. Let $(M,g)$ be a smooth Riemannian manifold and $H'$ a natural Hamiltonian of the form
    \[
   H' = \frac12 g^{-1}(p,p)-V(q).
   \]
    For $h> \sup_{q\in M} V(q)$, the Hamiltonian flow given by $\vec{H'}$ at energy $h$ on $T^*M$ is equivalent to the co-geodesic flow of the Jacobi-Maupertuis metric
   \begin{equation}
   	\label{eq:JMmetric}
     g_{JM}(v,v) = 2(h-V(q))g(v,v),
   \end{equation}
   on the unit cotangent bundle.
   Any integral curve of $\vec{H'}$ is a re-parametrization of a geodesic of $g_{JM}$ and vice versa. In what follows, we shall apply the Lagrangian version of the principle. The solutions of the Euler-Lagrange equations \eqref{eq:motion_equation} with energy $h$ are unparametrized geodesics of $g_{JM}$.
   
   For our choice of potential \eqref{eq:potential}, it turns that the Jacobi-Maupertuis metric $g_{JM}$ is asymptotically flat and negatively curved, as the following Lemma shows. Moreover, the line $\elld$ is concave when we pick the normal pointing towards the centres, as in Definition \ref{def:mirror}.
    \begin{lemma}[Curvatures]
    	Let $\gamma$ be a unit speed parametrization (with respect to the Jacobi-Maupertuis metric) of $\elld$ and let  $\tilde{\nabla}$ be the Levi-Civita connection. Then
    	\[
    	   \tilde{\nabla}_{\dot{\gamma}} \dot{\gamma} =   \frac{\langle 	\nabla V, v \rangle v-\nabla V }{4(h+V)^2}.
    	\]
    	In particular, taking the normal to $\elld$ pointing towards the centres, its curvature is negative. Moreover, the Gaussian curvature of the Jacobi-Maupertuis metric is 
    	\[
    	\kappa_{JM}= \frac{\Delta V}{2(h+V)}-\frac{\vert \nabla V \vert^2}{(2(h+V))^2}<0.
    	 \]
    	\begin{proof}
    		The part about the Gaussian curvature follows from \cite[Lemma 6.12]{baranzini2024chaotic}.    		
    		We now compute the covariant derivative $\tilde{\nabla}_{\dot{\gamma}}\dot{\gamma}$ with respect to the Jacobi-Maupertuis metric \eqref{eq:JMmetric}. Let us denote by $v,w$ two constant vector fields. Standard formulas for conformal changes of the metric (see for instance \cite[Theorem 7.30]{lee_book})  show that 
    	\[
    	\tilde{\nabla}_{v}v = \nabla_{v}v +  2 \langle \nabla \varphi,v\rangle  v -\nabla \varphi, \quad \text{where} \quad \varphi =\frac{1}{2} \log(2(h+V)).
    	\]
    	Since $\nabla \varphi = \frac{\nabla V}{2(h+V)}$, the covariant derivative reads
    	\[
    	\tilde{\nabla}_vv= \frac{\langle 	\nabla V, v \rangle }{h+V}v- \frac{\nabla V}{2(h+V)}
    	\]
    	Unit-speed parametrizations of $\elld$ are given by the flow lines of the vector field 
    	\[
    	X = f v =  \frac{v}{\sqrt{2(h+V)}}, \quad \tilde{\nabla}_XX = f^2\tilde{\nabla}_vv+f v(f) v.
    	\]		
    	Thus the covariant derivative reads
    	\begin{align*}
    		\tilde{\nabla}_{\dot{\gamma}}\dot{\gamma} &= \tilde{\nabla}_XX =  \frac{\langle 	\nabla V, v \rangle }{2(h+V)^2}v- \frac{\nabla V}{4(h+V)^2}-\frac{\langle \nabla V, v\rangle }{4(h+V)^2} v
    		\\ &=\frac{   \langle 	\nabla V, v \rangle v-\nabla V}{(2(h+V))^{2}}.
    	\end{align*}
    
     Let us observe that, if $w$ is orthogonal to $v$, then $\tilde{\nabla}_{\dot{\gamma}}\dot{\gamma} = -\frac{\langle \nabla V,w \rangle}{4(h+V)^2} w$ and thus the curvature Gaussian curvature of the boundary is negative if and only if $w$ points towards the centres. Indeed, choosing the origin on the line $\elld$ and parametrizing it as $t v $,  the Gaussian curvature reads
     \[
         {(2(h+V))^{\frac32}}\kappa_{\elld}(t) = -\langle \nabla V,w\rangle = - \sum_{i=1}^n\frac{ \alpha_i\langle c_i,w\rangle }{\vert t v -c_i \vert^{\alpha_i+2}}.
     \]     
     \end{proof}
 \end{lemma}

\subsection{Some properties of length-minimizers}
	
Let us fix $k \in \mathbb{Z}^* = \mathbb{Z}\setminus\{0\}$. Since we  will assume that there are at least two distinct centres, we may suppose that $c_1 \neq c_2$. The first task of this section is to build a foliation of length-minimizing curves with endpoints on $\elld$ winding $k-$times around $c_1$ and $c_2$, relatively to $\elld$. With a slight abuse of notation, we will denote by $[x,y]$ a curve parametrizing the segment with endpoints $x$ and $y$. 

We will say that a continuous curve $\eta: [a,b] \to \mathbb{R}^2\setminus \{c_1,\dots,c_n\}$ with $\eta(a),\eta(b) \in\elld$ winds  $k-$times around the centres if $\eta*[\eta(b),\eta(a)]$ is homotopic to a simple curve enclosing $c_1$ and $c_2$ run $k-$times. Here and throughout the paper the symbol $*$ stands for path concatenation.

\begin{definition}
		\label{def:minimizers}
		 Let $\sigma\subseteq \mathbb{R}^2 \setminus\{c_1,\dots,c_n\}$ be a simple curve enclosing only $c_1$ and $c_2$. For any $x,y \in \elld$, $k \in \mathbb{Z}^*$ let us denote by $\gamma_k^{x,y}$ any minimizer of the Jacobi-Maupertuis length in the weak closure of the space
		 \[
		  \mathcal{C}^{x,y}_k := \{ \eta \in H^1(\mathbb{R}^2\setminus \left\{ c_1,\dots c_n\right\}): \,\eta(0) = x,\, \eta(1) = y, \,  \eta*[y,x]\sim \sigma^k \},
		 \]		
		 i.e. the closure of paths from $x$ to $y$ winding $k-$times around $c_1$ and $c_2$.
		 Let us denote by $\gamma_k$ any length-minimizer in the weak closure of  \[\mathcal{C}_k:= \bigcup_{x,y\in \elld} \mathcal{C}_k^{x,y},\]		 
		 i.e. any minimizer among path with free endpoints on $\elld$ winding $k-$times.
	\end{definition}

	Of course, length minimizers in general might collide with the centres. However, our choice of space $\mathcal{C}_k^{x,y}$ guarantees that this is not the case. 
	\begin{lemma}
		\label{lemma:uniqueness_length_minimizers}
		The minimizers $\gamma_k$ and $\gamma_k^{x,y}$ are unique and completely contained in $\mathbb{R}^2\setminus \left\{c_1,\dots,c_n\right\}$.
	\begin{proof}
		The fact that minimizers are indeed collision-less can be proved applying the same local argument as in \cite[Section 4.2]{baranzini2024chaotic} or \cite[Proposition 5.1]{BolKoz2017}.		
		Uniqueness follows then from an application of Gauss-Bonnet theorem.
		
	    Let $\sigma$ a simple loop enclosing $c_1$ and $c_2$. Minimizers in the homotopy class of $\sigma$ may be collisional but, in any case, no $\gamma_{k}^{x,y}$ can intersect any such minimizers (compare with \cite[Proposition 3.12]{baranzini2024chaotic}). This implies that all $\gamma_{k}^{x,y}$ are contained in an open set $\mathcal{U}$ homeomorphic to $\mathbb{R}^2 \setminus [c_1,c_2]$.

		Let us assume that two distinct minimizers in $\mathcal{C}_k^{x,y}$ exists and denote them by $\eta_1$ and $\eta_2$. Let $\alpha_i$ be the turning angles between the velocities of $\eta_1$ and $\eta_2$ at $x$ and $y$.  Let us denote by $\Omega$ the region bounded by the two curves. Up to passing to a $k-$sheeted cover of $\mathcal{U} \sim \mathbb{R}^2 \setminus [c_1,c_2]$, we can assume that $\Omega$ is homeomorphic to a disk. Gauss-Bonnet theorem then implies that
		\[
		\int_{\Omega} \kappa_{JM}  =   2 \pi-\alpha_1-\alpha_2 .
		\]
		However, this is possible if and only if $\eta_1 = \eta_2$.
		
		The argument for minimizers in $\mathcal{C}_k$ is analogous. Assume that $\eta_1$ and $\eta_2$ are minimizers, with endpoints $x_i$ and $y_i$ respectively. Let $\Omega$ be the region bounded by $\eta_1$, $[y_1,y_2]$, $\eta_2^{-1}$ and $[x_2,x_1]$. 
		Since $\eta_i$ are minimizers, the velocities at the endpoints make a $\pi/2$ angle with $\elld$ and the sum of the turning angles is $2\pi$. It follows that 
		\[
			\int_{\Omega} \kappa_{JM}  = - \int_{\partial \Omega} \kappa_{\partial \Omega}.\]
		
		Thanks to Lemma \ref{lemma:convexity_mirror}, the curvature of the segments $[x_2,x_1]$ and $[y_1,y_2]$ is negative. Thus, the only possibility is again that $\eta_1$ and $\eta_2$ coincide.

	\end{proof}	 
    \end{lemma}

The last observation we make concerns the concavity of the line $\elld$. Roughly speaking, we can say that homotopically trivial length minimizers between two points $x,y \in \elld$ are completely contained in the half-plane without the centres $\overline{\Pi_{\elld}^c}$. Notice that this does not imply that minimizers $\gamma_{k}^{x,y}$ are contained in $\Pi_{\elld}$ as Figure \ref{fig:concavity_line} shows.
\begin{lemma}
	\label{lemma:convexity_mirror}
	Let $\eta:[0,T]\to \mathbb{R}^2\setminus\{c_1,\dots,c_n\}$ be a solution of \eqref{eq:motion_equation}. If $\eta(0)$, $\eta(T)\in \ell_d$ and $\dot{\eta}(0)$ points inside $\Pi_{\elld}$, then $\eta$ is homotopically non trivial, relatively to $\elld$.
	\begin{proof}
		The statement follows from an application of Gauss-Bonnet theorem. Since $\dot\eta(0)$ points towards the centres, $\eta(t)$ is contained in $\Pi_{\elld}$ for $t$ small. This means that the inner normal to the boundary of the disk $\Omega$ bound by $\eta$ and the segment $[\eta(0),\eta(T)]$ points toward the centers. In particular, the segment $[\eta(0),\eta(T)]$ has negative Gaussian curvature. Let $\alpha_0$ and $\alpha_T$ be the turning angles at the endpoints of $\eta$, we have
		\[
		\int_\Omega \kappa_{JM} + \int_{\partial \Omega} \kappa_\eta = 2\pi -(\alpha_0+\alpha_1).
		\]
		However the right hand side is non negative whereas the left one is strictly negative, a contradiction.
	\end{proof}
\end{lemma}
\begin{figure}
	\label{fig:concavity_line}
	\caption{The sign of the curvature depends on the orientation of the boundary. The drawing on the left does not contradict Gauss Bonnet theorem. On the right we are in the situation of Lemma \ref{lemma:convexity_mirror}}

\tikzset{
	pattern size/.store in=\mcSize, 
	pattern size = 5pt,
	pattern thickness/.store in=\mcThickness, 
	pattern thickness = 0.3pt,
	pattern radius/.store in=\mcRadius, 
	pattern radius = 1pt}
\makeatletter
\pgfutil@ifundefined{pgf@pattern@name@_qavrjy5sl}{
	\pgfdeclarepatternformonly[\mcThickness,\mcSize]{_qavrjy5sl}
	{\pgfqpoint{0pt}{-\mcThickness}}
	{\pgfpoint{\mcSize}{\mcSize}}
	{\pgfpoint{\mcSize}{\mcSize}}
	{
		\pgfsetcolor{\tikz@pattern@color}
		\pgfsetlinewidth{\mcThickness}
		\pgfpathmoveto{\pgfqpoint{0pt}{\mcSize}}
		\pgfpathlineto{\pgfpoint{\mcSize+\mcThickness}{-\mcThickness}}
		\pgfusepath{stroke}
}}
\makeatother


\tikzset{
	pattern size/.store in=\mcSize, 
	pattern size = 5pt,
	pattern thickness/.store in=\mcThickness, 
	pattern thickness = 0.3pt,
	pattern radius/.store in=\mcRadius, 
	pattern radius = 1pt}
\makeatletter
\pgfutil@ifundefined{pgf@pattern@name@_63zzv7q6p}{
	\pgfdeclarepatternformonly[\mcThickness,\mcSize]{_63zzv7q6p}
	{\pgfqpoint{0pt}{-\mcThickness}}
	{\pgfpoint{\mcSize}{\mcSize}}
	{\pgfpoint{\mcSize}{\mcSize}}
	{
		\pgfsetcolor{\tikz@pattern@color}
		\pgfsetlinewidth{\mcThickness}
		\pgfpathmoveto{\pgfqpoint{0pt}{\mcSize}}
		\pgfpathlineto{\pgfpoint{\mcSize+\mcThickness}{-\mcThickness}}
		\pgfusepath{stroke}
}}
\makeatother
\tikzset{every picture/.style={line width=0.75pt}} 

\begin{tikzpicture}[x=0.75pt,y=0.75pt,yscale=-1,xscale=1]
	
	\draw [line width=1.5]    (126,49) -- (222.33,209) ;
	\draw [line width=1.5]  [dash pattern={on 5.63pt off 4.5pt}]  (108.67,19) -- (126,49) ;
	\draw [line width=1.5]  [dash pattern={on 5.63pt off 4.5pt}]  (222.33,209) -- (239.67,239) ;
	\draw  [fill={rgb, 255:red, 155; green, 155; blue, 155 }  ,fill opacity=1 ] (201,72) .. controls (201,70.34) and (202.34,69) .. (204,69) .. controls (205.66,69) and (207,70.34) .. (207,72) .. controls (207,73.66) and (205.66,75) .. (204,75) .. controls (202.34,75) and (201,73.66) .. (201,72) -- cycle ;
	\draw  [fill={rgb, 255:red, 155; green, 155; blue, 155 }  ,fill opacity=1 ] (231,62) .. controls (231,60.34) and (232.34,59) .. (234,59) .. controls (235.66,59) and (237,60.34) .. (237,62) .. controls (237,63.66) and (235.66,65) .. (234,65) .. controls (232.34,65) and (231,63.66) .. (231,62) -- cycle ;
	\draw    (137,70) .. controls (177,40) and (248.33,35) .. (246.33,59) .. controls (244.33,83) and (207.33,97) .. (192.33,86) .. controls (177.33,75) and (202.33,24) .. (236.33,30) .. controls (270.33,36) and (258.33,105) .. (190.33,155) ;
	\draw  [pattern=_qavrjy5sl,pattern size=6pt,pattern thickness=0.75pt,pattern radius=0pt, pattern color={rgb, 255:red, 0; green, 0; blue, 0}] (163.33,206) .. controls (156.33,175) and (163.33,172) .. (190.33,155) .. controls (202.33,177) and (204.33,177) .. (222.33,209) .. controls (198.33,233) and (172.33,245) .. (163.33,206) -- cycle ;
	\draw  [fill={rgb, 255:red, 0; green, 0; blue, 0 }  ,fill opacity=1 ] (185.07,236.22) -- (196.06,230.12) -- (184.59,224.98) -- (190.44,230.36) -- cycle ;
	\draw  [fill={rgb, 255:red, 0; green, 0; blue, 0 }  ,fill opacity=1 ] (207.92,174.07) -- (198.06,166.27) -- (197.73,178.84) -- (200.44,171.36) -- cycle ;
	\draw [color={rgb, 255:red, 155; green, 155; blue, 155 }  ,draw opacity=1 ][line width=1.5]    (460,175) -- (490.83,157.94) ;
	\draw [shift={(494.33,156)}, rotate = 151.04] [fill={rgb, 255:red, 155; green, 155; blue, 155 }  ,fill opacity=1 ][line width=0.08]  [draw opacity=0] (11.61,-5.58) -- (0,0) -- (11.61,5.58) -- cycle    ;
	\draw [line width=1.5]    (385,50) -- (481.33,210) ;
	\draw [line width=1.5]  [dash pattern={on 5.63pt off 4.5pt}]  (367.67,20) -- (385,50) ;
	\draw [line width=1.5]  [dash pattern={on 5.63pt off 4.5pt}]  (481.33,210) -- (498.67,240) ;
	\draw  [fill={rgb, 255:red, 155; green, 155; blue, 155 }  ,fill opacity=1 ] (524.54,120.14) .. controls (525.28,118.67) and (527.09,118.07) .. (528.57,118.82) .. controls (530.05,119.57) and (530.64,121.37) .. (529.89,122.85) .. controls (529.14,124.33) and (527.34,124.92) .. (525.86,124.18) .. controls (524.38,123.43) and (523.79,121.62) .. (524.54,120.14) -- cycle ;
	\draw  [fill={rgb, 255:red, 155; green, 155; blue, 155 }  ,fill opacity=1 ] (555.82,124.76) .. controls (556.57,123.28) and (558.37,122.68) .. (559.85,123.43) .. controls (561.33,124.18) and (561.92,125.98) .. (561.18,127.46) .. controls (560.43,128.94) and (558.62,129.53) .. (557.14,128.79) .. controls (555.67,128.04) and (555.07,126.23) .. (555.82,124.76) -- cycle ;
	\draw    (420.36,107.63) .. controls (390.33,134) and (372.33,118) .. (364.33,85) .. controls (356.33,52) and (456.33,50) .. (508.33,55) .. controls (560.33,60) and (593.33,102) .. (570.86,129) .. controls (548.38,155.99) and (518.91,145.31) .. (510.49,128.73) .. controls (502.06,112.14) and (541.33,89) .. (575.02,98.6) .. controls (608.7,108.21) and (552.58,175.06) .. (469.33,189) ;
	\draw  [pattern=_63zzv7q6p,pattern size=6pt,pattern thickness=0.75pt,pattern radius=0pt, pattern color={rgb, 255:red, 0; green, 0; blue, 0}] (479.36,110.64) .. controls (486.36,141.64) and (479.36,144.64) .. (452.35,161.63) .. controls (440.36,139.63) and (438.36,139.63) .. (420.36,107.63) .. controls (444.37,83.63) and (470.37,71.64) .. (479.36,110.64) -- cycle ;
	\draw  [fill={rgb, 255:red, 0; green, 0; blue, 0 }  ,fill opacity=1 ] (487.46,127.56) -- (481.03,116.75) -- (476.25,128.38) -- (481.44,122.36) -- cycle ;
	\draw  [fill={rgb, 255:red, 0; green, 0; blue, 0 }  ,fill opacity=1 ] (426.62,130.1) -- (437,137.19) -- (436.44,124.63) -- (434.26,132.28) -- cycle ;
	\draw [color={rgb, 255:red, 155; green, 155; blue, 155 }  ,draw opacity=1 ][line width=1.5]    (179,136) -- (146.9,152.2) ;
	\draw [shift={(143.33,154)}, rotate = 333.22] [fill={rgb, 255:red, 155; green, 155; blue, 155 }  ,fill opacity=1 ][line width=0.08]  [draw opacity=0] (11.61,-5.58) -- (0,0) -- (11.61,5.58) -- cycle    ;
	
	\draw (145.17,119.9) node [anchor=north west][inner sep=0.75pt]    {$w$};
	\draw (474.17,165.9) node [anchor=north west][inner sep=0.75pt]    {$w$};
	\draw (228,152.4) node [anchor=north west][inner sep=0.75pt]    {$\kappa _{\ell }  >0$};
	\draw (376,142.4) node [anchor=north west][inner sep=0.75pt]    {$\kappa _{\ell } < 0$};

\end{tikzpicture}
\end{figure}
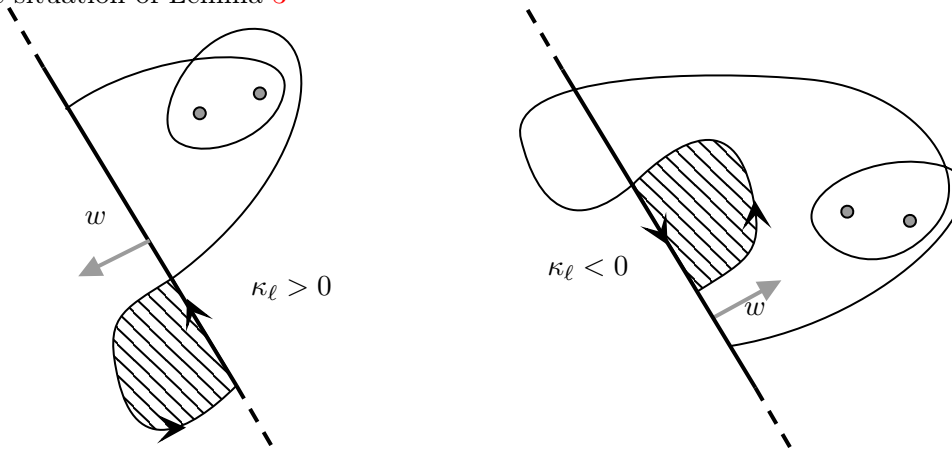

\section{The high - energy/distant mirror regime}
\label{section:perturbative_version}

Let us pick two distinct integers $k,m\in \mathbb{Z}^*$. Thanks to Lemma \ref{lemma:uniqueness_length_minimizers}, for any $x,y\in \elld$ we have a unique path $\gamma_k^{x,y}$ (and $\gamma_m^{x,y}$ respectively) joining $x$ to $y$. Let us consider a finite sequence $(i_1, \dots, i_r)$ such that $i_j\in\{k,m\}$ and define the following family of closed paths
\begin{equation}
	\label{eq:def_space_paths}
\mathcal{C}_{(i_1, \dots, i_r)} = \{\gamma_{i_1}^{x_1,x_2}*\gamma_{i_2}^{x_2,x_3}*\dots*\gamma_{i_r}^{x_{r},x_1}: x_j\in\elld\}.
\end{equation}
Here as before $\gamma*\eta$ stands for the concatenation of paths. We now look for minimizers of the Jacobi-Maupertuis length in the space $\mathcal{C}_{(i_1,\dots,i_r)}$ and prove the following

\begin{thm}
	\label{thm:existence_periodic}
	For any $r>1$ and  $(i_1,\dots,i_r)$, minimizers exists. Moreover, there exist $d_0(v,h)>0$ such that, for any $k,m$ and $d\ge d_0(v,h)$ minimizers are billiard trajectories, i.e. completely contained in the half-plane $\Pi_{\elld}$.	
\end{thm}

A few remarks on the Theorem statement are in order. The existence of minimizers and the fact that they satisfy the elastic reflection law are fairly straightforward. The crucial point is to show that these minimizers lie entirely in the half-plane $\Pi_{\elld}$. This is because $\Pi_{\elld}$ is not convex with respect to the Jacobi-Maupertuis metric and the paths $\gamma_k^{x,y}$ need not lie completely in a single half-space for all choices of $x,y \in \elld$.

Let us consider now the space of sequences $\{k,m\}^\mathbb{Z}$.  The next Theorem states that we can associate to any sequence $\{i_j\}_{j \in \mathbb{Z}}$ a billiard trajectory obtained concatenating arcs winding $i_j-$times around the centres. Let $\xi:\mathbb{R}\to \mathbb{R}^2\setminus \{c_1,\dots, c_n\}$ be a billiard trajectory, assume that $\xi(0) \in \elld$ and denote by $\xi_i, i \in \mathbb{Z}$ the restriction of $\xi$ between two consecutive bounces. We have the following Theorem.
\begin{thm}
	\label{thm:non_periodic_orbits}
	For $d\ge d_0(v,h)$ as in Theorem \ref{thm:existence_periodic} and any sequence $\{i_j\}_{j \in \mathbb{Z}}\in \{k,m\}^\mathbb{Z}$ there exists a billiard trajectory $\xi$ such that its restriction between consecutive bouncing points $\xi_j$ belongs to $\mathcal{C}_{i_j}$.
\end{thm}

In Section \ref{subsec:transversality} we give some compactness and transversality argument needed in the proof of Theorem \ref{thm:existence_periodic} and \ref{thm:non_periodic_orbits}. The proof of the Theorems are given in Section \ref{subsec:proof_thm_perturbative}. 
\subsection{Compactness and Lagrange-Jacobi inequality}
\label{subsec:transversality}
Now we investigate some qualitative features of the $\gamma_k^{x,y}$, for $k$  in $\mathbb{Z}^*$.
The main properties we aim to prove are the following. 
\begin{lemma}
	\label{lemma:compact_set}
	 There exists $d_0>0$ such that for all $d>d_0$ there exists a compact segment $s_d$ of $\elld$ and a compact region $\Omega$ enclosing the centres such that
	 \begin{itemize}
	 	\item  $\elld \cap \Omega = \emptyset$,
	 	\item for all $k$ and $x,y \in s_d$, $\gamma_k$ and $\gamma_k^{x,y}$ intersects $\partial \Omega$ exactly twice,
	 	\item for all $k$ and $x,y \in s_d$, the portion of $\gamma_k$ and $\gamma_k^{x,y}$ contained in $\Omega$ is homotopic, relatively to $\partial \Omega$, to a simple curve enclosing $c_1$ and $c_2$ run $k-$times,
	 	\item for all $(i_0,\dots,i_r)$ the bouncing points of periodic minimizers in $\mathcal{C}_{(i_0,\dots, i_r)}$ lie in $s_d$.
	 \end{itemize}
\end{lemma}

Loosely speaking, this means that the $\gamma_k^{x,y}$ with initial conditions in $s_d$ join $\elld$ to $\partial \Omega$, wind $k-$times around $c_1$ and $c_2$ inside $\Omega$ and then leave $\Omega$ to reach $\elld$ again. 
The proof of Lemma \ref{lemma:compact_set} relies on the following compactness property of length-minimizers.

\begin{figure}
	\centering
	\resizebox{\textwidth}{!}{
	\input{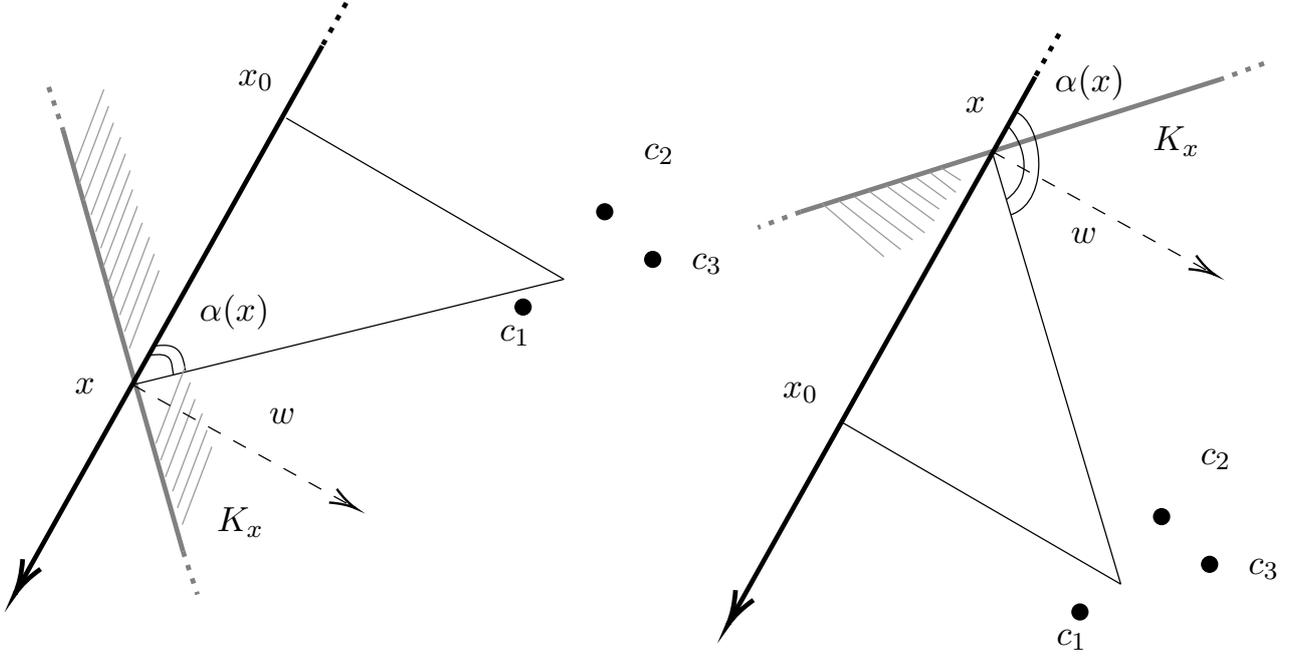}}
	\caption{The cones $K_x$ appearing in Lemma \ref{lemma:minimizers_in_a_ball} }
	\label{fig:cones}
\end{figure}

\begin{lemma}
	\label{lemma:minimizers_in_a_ball}
	Let us denote by $x_0$ the orthogonal projection of the origin to $\elld$. There exists a $R = R(d)>0$ such that, for all $k$, the $\gamma_k$ and $\gamma_k^{x,y}$ with $x,y \in B_R(x_0)$ as well as minimizers in $\mathcal{C}_{(i_1,\dots, i_r)}$ for any $(i_1,\dots,i_r)$ are completely contained in $B_R(x_0)$. 
	
	Moreover, fixing the origin at the barycentre of the $c_i$'s, we have that the angle between the initial velocities of minimizers $\dot\gamma_k^{x,y}(0)$  and the normal $w$ of $\ell_d$ always belong to the $\pi$-cone given by 
	\[
	K_x = \left(\alpha(x), \pi -\alpha(x)\right), \quad \alpha(x) =\begin{cases}
		\pi -\arcsin(\frac{\vert x_0\vert}{\vert x \vert}), &x>x_0\\
		\arcsin(\frac{\vert x_0\vert}{\vert x \vert}), &x\le x_0.		
	\end{cases}
	\]
	\begin{proof}
		Let us change coordinates and fix the origin at the point $x_0$. Let $R>0$ and let us consider the family of dilations $\delta_s(x) = e^{-s \psi(x)} x$, where $\psi(x) = \psi(\vert x\vert)$ is a radial monotone function, vanishing inside $B_R$ and equal to $1$ outside $B_{2R}$. Let $\gamma$ be a length minimizer and consider the family of curves $\delta_s(\gamma)$. The derivatives with respect to $t$ and $s$ respectively read
		\[
		\frac{d}{dt}\delta_s(\gamma) = -s \langle \nabla \psi, \dot{\gamma}\rangle \delta_s(\gamma)+\delta_s(\dot{\gamma}), \quad \frac{d^2}{dt ds}\delta_s(\gamma)\vert_{s=0} =  -\langle \nabla \psi, \dot{\gamma}\rangle \gamma - \psi  \dot\gamma.
		\]
		Similarly, the derivative with respect to $s$ at $s=0$ is
		\[
		\frac{d}{d s} \delta_s(\gamma) \vert_{s=0}= -\psi \gamma.
		\]
		Thus, the first variation of the length reads:
		\begin{align*}
			\frac{d \ell(\delta_s(\gamma)}{ds}\vert_{s=0} &= \int - \langle \langle \nabla \psi, \dot{\gamma}\rangle \gamma + \psi  \dot\gamma, \dot{\gamma}\rangle \frac{\sqrt{h+V}}{\vert\dot{\gamma}\vert}-\frac{\vert\dot{\gamma}\vert\langle \nabla V,\psi \gamma \rangle }{2\sqrt{h+V}} \\
			&= \int-\psi'(\vert \gamma \vert)\left(\frac{\langle\dot{\gamma},\gamma \rangle} {\vert \dot\gamma \vert}\right)^2\sqrt{h+V}-\frac{ \psi\vert\dot{\gamma}\vert}{\sqrt{h+V }} (h+V+\frac12\langle\nabla V,\gamma\rangle).
		\end{align*}
		If $R$ is big enough, the quantity $(h+V+\frac12\langle \nabla V,\gamma\rangle)$ is positive outside the ball of radius $R$ (see Remark \ref{remark:ball_half_plane}). This inequality is customarily called \emph{Lagrange-Jacobi inequality}.  So, if $\gamma$ were not completely contained in $B_R$, the derivative above would be strictly negative and $\gamma$ would not be a length-minimizer. This variation is admissible whenever $\gamma$ is an arc with endpoints in $B_R(x_0)$, a minimizer with free endpoints in $\elld$ or a periodic billiard trajectory.
		
		The second part of the statement follows again from the inequality $h+V+\frac{1}{2}\langle \nabla V, x\rangle \ge0$. Pick as origin the barycentre of the $c_i$s. Outside a ball of radius $R_0$ (independent of $d$) the Lagrange-Jacobi inequality holds. Thus, along solutions of \eqref{eq:motion_equation}, the distance from the origin is increasing provided that the initial velocity of solutions of \eqref{eq:motion_equation} satisfies $\langle \gamma,\dot{\gamma}\rangle \ge0$. Thus, initial conditions in the cones specified above correspond to solutions going to spatial infinity which never meet $\elld$ again.
	\end{proof}
\end{lemma}

\begin{rem}[Lagrange-Jacobi inequality]
	\label{remark:ball_half_plane}
	First of all let us observe that 
	\begin{align}
		\nonumber
		h+V+\frac12\langle \nabla V,x\rangle &= h +\sum_{i=1}^{n}\frac{m_i}{\vert x-c_i\vert^{\alpha_i}}-\frac{m_i \alpha_i\langle x,x-c_i\rangle}{2\vert x-c_i\vert^{\alpha_i+2}}\\
		&\ge h +\sum_{i=1}^{n}\frac{m_i}{\vert x-c_i\vert^{\alpha_i}}\left(1-\frac{\alpha_i\vert x\vert}{2\vert x-c_i\vert}\right).
		\label{eq:lagrange_jacobi}
	\end{align}
	If $\alpha_i<2$ for all the $i$, $2\vert x-c_i\vert\ge \alpha_i \vert x\vert$ outside of a circle of radius $\frac{2}{2-\alpha_i} \vert c_i\vert$. Thus the quantity $h+V+\frac12\langle \nabla V,\gamma\rangle$ is positive outside a ball of radius $R_0 = \max_{i=1,\dots,n} \frac{2\vert c_i\vert}{2-\alpha_i}$, independently on the energy. If some of the $\alpha_i\ge 2$ the second term in the equation above is always negative but tends to zero at infinity.
	
	Moreover, fixing the origin at the barycentre of the $c_i$, one obtains that the ball $B_R$ of the previous Lemma can be chosen to lie in the half plane $\{x:\langle x, w\rangle \le R_0 \}$ which is independent of the distance between the line $\elld$ and the centres, see Figure \ref{fig:omega}.
\end{rem}

   \begin{figure}
   	\centering
   	\tikzset{every picture/.style={line width=0.75pt}} 

\begin{tikzpicture}[x=0.75pt,y=0.75pt,yscale=-1,xscale=1]
	
	\draw [line width=1.5]  [dash pattern={on 1.69pt off 2.76pt}]  (83,57) -- (105,79) ;
	\draw [line width=1.5]  [dash pattern={on 1.69pt off 2.76pt}]  (321,295) -- (334.17,308.17) -- (343,317) ;
	\draw  [fill={rgb, 255:red, 0; green, 0; blue, 0 }  ,fill opacity=1 ] (267,97) .. controls (267,95.9) and (267.9,95) .. (269,95) .. controls (270.1,95) and (271,95.9) .. (271,97) .. controls (271,98.1) and (270.1,99) .. (269,99) .. controls (267.9,99) and (267,98.1) .. (267,97) -- cycle ;
	\draw  [fill={rgb, 255:red, 0; green, 0; blue, 0 }  ,fill opacity=1 ] (261,141) .. controls (261,139.9) and (261.9,139) .. (263,139) .. controls (264.1,139) and (265,139.9) .. (265,141) .. controls (265,142.1) and (264.1,143) .. (263,143) .. controls (261.9,143) and (261,142.1) .. (261,141) -- cycle ;
	\draw [color={rgb, 255:red, 155; green, 155; blue, 155 }  ,draw opacity=1 ]   (372.33,8) -- (83.33,297) ;
	\draw  [color={rgb, 255:red, 155; green, 155; blue, 155 }  ,draw opacity=1 ][line width=1.5]  (112.91,89.79) .. controls (165.71,38.59) and (250.02,39.89) .. (301.23,92.69) .. controls (352.43,145.49) and (351.13,229.8) .. (298.33,281) .. controls (245.53,332.2) and (161.22,330.91) .. (110.02,278.11) .. controls (58.82,225.31) and (60.11,141) .. (112.91,89.79) -- cycle ;
	\draw  [color={rgb, 255:red, 155; green, 155; blue, 155 }  ,draw opacity=1 ][dash pattern={on 4.5pt off 4.5pt}] (220.33,120) .. controls (220.33,95.7) and (240.03,76) .. (264.33,76) .. controls (288.63,76) and (308.33,95.7) .. (308.33,120) .. controls (308.33,144.3) and (288.63,164) .. (264.33,164) .. controls (240.03,164) and (220.33,144.3) .. (220.33,120) -- cycle ;
	\draw [color={rgb, 255:red, 155; green, 155; blue, 155 }  ,draw opacity=1 ] [dash pattern={on 4.5pt off 4.5pt}]  (264.33,120) -- (308.33,120) ;
	\draw [line width=1.5]    (340,40) .. controls (345,36.25) and (316.33,159) .. (284.33,179) .. controls (252.33,199) and (234.33,175) .. (227.83,152.5) .. controls (221.33,130) and (229.33,115) .. (237.33,103) .. controls (245.33,91) and (262.33,80) .. (340,40) -- cycle ;
	\draw [color={rgb, 255:red, 155; green, 155; blue, 155 }  ,draw opacity=1 ] [dash pattern={on 4.5pt off 4.5pt}]  (203.33,177) -- (258,308) ;
	\draw [line width=1.5]    (105,79) -- (321,295) ;
	\draw [color={rgb, 255:red, 155; green, 155; blue, 155 }  ,draw opacity=1 ]   (437.33,227) -- (215.33,8) ;
	
	\draw (281,99) node [anchor=north west][inner sep=0.75pt]   [align=left] {$\displaystyle R_{0}$};
	\draw (247.33,90) node [anchor=north west][inner sep=0.75pt]   [align=left] {$\displaystyle c_{1}$};
	\draw (272.33,132) node [anchor=north west][inner sep=0.75pt]   [align=left] {$\displaystyle c_{2}$};
	\draw (337,66) node [anchor=north west][inner sep=0.75pt]   [align=left] {$\displaystyle \gamma _{r}$};
	\draw (216,254) node [anchor=north west][inner sep=0.75pt]   [align=left] {$\displaystyle R$};
	\draw (137,85.4) node [anchor=north west][inner sep=0.75pt]    {$\ell _{v,d}$};

\end{tikzpicture}
   	\caption{The region $\Omega$ is obtained as the intersection of the interior of $\gamma_r$ and the half plane supporting $B_R$.}
   	\label{fig:omega}
   \end{figure}
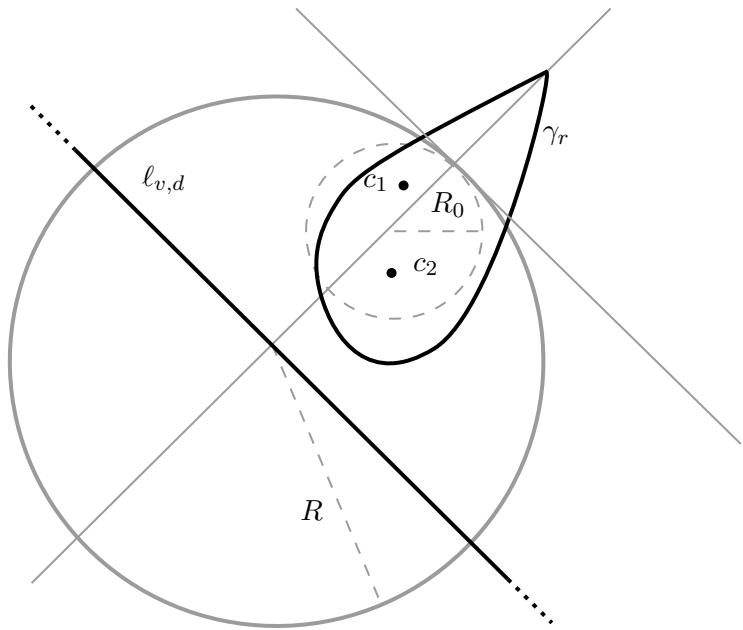

 \begin{proof}[Proof of Lemma \ref{lemma:compact_set}]
 	Let us pick as origin the barycentre of the $c_i$s and let $r>0$. Consider $B_r(0)$ the ball of radius $r$ centred at the origin. Let us consider the (unique up to reparametrization) curve $\gamma_r : [0,1] \to \mathbb{R}^2\setminus\{c_1,\dots,c_n\}$ which is homotopic to $\partial B_r$, satisfies
 	\[
 	\gamma_r(0) = \gamma_r(1) = r w
 	\]
 	and is length-minimizing. Denote by $\Omega_r $ the region bound by $\gamma_r$. Clearly,  $\Omega_r \subseteq \Omega_{r'}$ if $r \le r'$ since length minimizer cannot intersect.
 	
    Thanks to Lemma \ref{lemma:minimizers_in_a_ball} and Remark \ref{remark:ball_half_plane}, we can take $r>R_0$ so that the endpoints of $\gamma_r$ lie outside of $B_R(dw)$. Here $R$ is the one given in Lemma \ref{lemma:minimizers_in_a_ball}. Moreover, for a fixed choice of $r$ there exists $d_0$ such that for all $d> d_0$, $\elld$ does not intersect $\Omega_r$. On the other hand, $B_R\setminus \Omega_r$ is contractible and thus every $\gamma_k$ and $\gamma_k^{x,y}$ has to enter (and exit) $\Omega_r$. Thanks to minimality, the intersection with $\partial \Omega_r$ are at most two. For the same reason, the closed curve obtained as concatenation of the portion of $\gamma_k^{x,y}$ inside $\Omega$ with an arc of $\partial \Omega$ must be homotopic to a simple curve enclosing $c_1$ and $c_2$ run $k-$times. Defining 
    \[
    \Omega := \{x: \langle w, x\rangle \le R_0\} \cap \Omega_r \supset B_R\cap \Omega_r
\] we have the desired compact set and $s_d$ can be taken as $B_R\cap \elld$.
 \end{proof}

\subsection{Proof of Theorem \ref{thm:existence_periodic} and \ref{thm:non_periodic_orbits}}
\label{subsec:proof_thm_perturbative}
We are now in a position to prove Theorem \ref{thm:existence_periodic} and \ref{thm:non_periodic_orbits}.
\begin{proof}[Proof of Theorem \ref{thm:existence_periodic}]
	We divide the proof in two steps. First, we prove existence of minimizers in $\mathcal{C}_{(i_1,\dots,i_r)}$ (as defined in  \eqref{eq:def_space_paths}) and show that they satisfy the bouncing condition. Then, we show that they are completely contained in a half plane and thus billiard trajectories.
	
	\textbf{Step 1.} Existence of minimizers in $\mathcal{C}_{(i_1,\dots,i_r)}$ is straightforward. Indeed, the bouncing points (i.e. endpoints of the $\gamma_k^{x,y}$ and $\gamma_m^{x,y}$) must lie in a compact segment $s_d$ of $\elld$ thanks to Lemma \ref{lemma:minimizers_in_a_ball} (otherwise the same variation used there decreases length). Thus, we are minimizing on a compact set diffeomorphic to a closed $r-$dimensional cube and minimizers always exists.
	
	To see that the elastic reflection condition holds, we can argue as follows.  Let us consider a critical point of the JM length in  $\mathcal{C}_{(i_1,\dots,i_r)}$, assume that the point $x\in \elld$ is a bouncing point and consider the portion of the minimizer given by
	\[
	\gamma_{i_1}^{y,x}*\gamma_{i_2}^{x,z}, \quad y,z \in \elld \,\, i_j \in \{k,m\}.
	\]
	For simplicity set $T_1 = T_{i_1}^{y,x}$ and $T_2 = T_{i_2}^{x,z}$ and consider
	\[
	(\xi_1,\xi_2) = (	\gamma_{i_1}^{y,x}(t),\gamma_{i_2}^{x,z}(T_2-t)) \]
	which is an element of the space
	\[	    \mathcal{C}_{y,z}=\{(\eta_1,\eta_2): \eta_i \in H^1[0,T_i], \eta_1(0) = y, \eta_2(0)=z, \eta_1(T_1) = \eta_2(T_2), \eta_j\sim \gamma_{i_j}\}.
	\]
	Differentiating the length functional at $(\xi_1,\xi_2)$ in $\mathcal{C}_{y,z}$ one finds that
	\[
	\langle \dot\xi_1(T_1)+\dot\xi(T_2),v\rangle = \langle \dot\gamma_{i_1}^{y,x}(T_1)-\dot \gamma_{i_2}^{x,z}(0)) ,v\rangle = 0.  
	\]
	Since the line
	\[
	\{\gamma_{i_1}^{y,w}*\gamma_{i_2}^{w,z}, \, w\in \elld\}
	\]
	can be embed into $\mathcal{C}_{y,z}$ without changing the length, the claim follows.
	
	\textbf{Step 2. } 	We show now that, at bouncing points, the velocities of the minimizers point inwards the half-plane $\Pi_{\elld}$. This is a consequence of Lemma \ref{lemma:compact_set}. In fact, assume as before that  $x\in \elld$ is a bouncing point and consider the portion of minimizer given by
	\[
	\gamma_{i_1}^{y,x}*\gamma_{i_2}^{x,z}, \quad y,z \in \elld \,\, i_j \in \{k,m\}.
	\]
	By Lemma \ref{lemma:compact_set}, the initial velocity $\gamma_{i_2}^{x,z}$ belongs to the cone $K_x$ and the final velocity of $\gamma_{i_1}^{y,x}$ to $-K_x$ since neither of them escapes to infinity. Let us denote by $R_{v}$ the reflection induced by the billiard law. It follows that the initial velocity of $\gamma_{i_2}^{x,z}$ must lie in $K_x\cap R_{v}(-K_x)$. This cone is easily seen to be completely contained in the half-plane $\Pi_{\elld}$. 
	
	We have proved that for short times the arcs of a minimizer $\gamma$ lie in the region containing the centres. Let us fix an arc $\gamma_j^{x,y}$ and assume that it intersects $\elld$ in a point $z \ne x,y$. Without loss of generality let us assume that this is the first one. Thanks to Lemma \ref{lemma:convexity_mirror}, we know that the portion of $\gamma_j^{x,y}$ joining $x$ to $z$, which we will call $\eta$, cannot be homotopically trivial. Thanks to Lemma \ref{lemma:compact_set}, $\eta$ winds $j-$times around the centres inside $\Omega$ and then intersects the line $\elld$. Since the final velocity at $y$ must lie in $-K_y\cap R_{v}(K_y)$, there should be a fourth intersection point and a homotopically trivial segment joining it to $y$, contradicting Lemma \ref{lemma:convexity_mirror}.	
	Thus, we have proved that minimizers $\gamma$ are completely contained in the half-space with the centres and so billiard trajectories.
\end{proof}

\begin{proof}[Proof of Theorem \ref{thm:non_periodic_orbits}]
   We will build the desired obits as limits of periodic ones. First we need the following observation. As already noticed, for any finite sequence $(i_1,\dots,i_r)$ the bouncing point of minimizers in $\mathcal{C}_{(i_1,\dots,i_r)}$ lie in a common compact interval $s_d\subset\elld$, thanks to Lemma \ref{lemma:minimizers_in_a_ball}. Since the restriction of the flow is continuous and the time between consecutive bounces bounded, the sets $\cup_{x,y \in s_d}\gamma_k^{x,y}$ and $\cup_{x,y \in s_d}\gamma_m^{x,y}$ are closed in $\mathbb{R} \setminus \{c_1,\dots, c_n\}$. Thus, the minimizers never approach collisions. 
   
   Parametrize the $\gamma_{i_j}^{x,y}$ as solutions of \eqref{eq:motion_equation}. The energy equation
   \[
   h = \frac{1}{2}\vert \dot\gamma_{i_j}^{x,y}\vert - V(\gamma_{i_j}^{x,y})
   \]
   implies that the $\gamma_{i_j}^{x,y}$ are pre-compact in the $C^0$ topology. Let us consider a sequence $\{i_j\}_{j \in \mathbb Z}$ and for $n \in \mathbb{N}$ the finite subsequences $(i_{-n},i_{-n+1}, \dots, \dots, i_n)$. Let us denote by $\eta_n$ a minimizer in $\mathcal{C}_{(i_{-n},i_{-n+1}, \dots, \dots, i_n)}$ and let us assume that it is extended by periodicity to a curve on the whole real line. This would be enough to prove convergence in the $C^0[-T,T]$ topology for any $T$. However, since the curve $\eta_n$ are just piecewise $C^2$ we cannot apply the usual bootstrap argument. We can however resort to the following unfolding procedure. Let $R_{\elld}$ the reflection about $\elld$. We consider the $C^1$ curves $\tilde\eta_n$ obtained reflecting the $\eta_n$ after each intersection with the wall $\elld$. Notice that, thanks to the motion equation \eqref{eq:motion_equation}, the first derivatives are now Lipschitz function, with an uniform Lipschitz constant.
   Thus, we can apply Ascoli-Arzelà Theorem and produce, for any interval $[-T,T]$ a subsequence converging in the $C^1$ topology to a curve $\eta$.  Using a standard diagonal argument we can find a subsequence converging uniformly to a curve $\eta$ on every compact interval.  Now, thanks to Lemma \ref{lemma:minimizers_in_a_ball} and the observation about the admissible cones in the proof of Theorem \ref{thm:existence_periodic}, we know that the velocities $\dot{\eta}_n$ lie in a cone not containing the direction $v$ and thus are never tangent to $\elld$, with possibly the exception of one point of $\elld$. Thus the set $\eta^{-1}(\elld)$ is a discrete set (recall that the convergence is $C^1$). For consecutive $a, b \in \eta^{-1}(\elld)$ we can therefore show, by the usual bootstrap argument, that the convergence is actually in the $C^2$ topology on compacts $[c,d]\subset (a,b)$ and thus the limit $\eta$ is piecewise a solution of \eqref{eq:motion_equation}.
\end{proof}

\section{The $2-$centres case}
\label{section:2_centres}
In this section, we focus on the problem with two Newtonian centres. It is well known that this system is integrable in elliptic-hyperbolic coordinates, as recalled in below and outlined in the Appendix \ref{appendix:2_centres}. Using the additional information about the system, we go further the perturbative setting of Theorem \ref{thm:existence_periodic} and \ref{thm:non_periodic_orbits}. 

\subsection{The curves $s_{\pm\infty}$}
\label{subsec:curve_infinite_turns}
Recall that the  Hamiltonian of the $2-$centres problem is given by
\[
    H (p,x) = \frac{1}{2}\vert p \vert^2+\frac{m_1}{\vert x-c_1\vert}+\frac{m_2}{\vert x-c_2\vert}.
\]
Choosing as origin the midpoint of $[c_1,c_2]$, and after a rotation and scaling, we can assume that $c_1 = -c_2 = (1,0)$.  We set (compare with Appendix \ref{appendix:2_centres})
\[
x = \Phi(\xi,\eta) = (\cosh \xi \cos \eta, \sinh \xi \sin \eta)
\]
and extend the transformation to the cotangent bundle accordingly. In these new coordinates, the conserved quantity of the system reads
\begin{align*}
      I &= \frac{1}{2} \dot{\xi}^2 -(m_1+m_2)\cosh \xi-H \cosh^2 \xi \\&= - \frac12 \dot{\eta}^2+(m_1-m_2)\cos \eta -H \cos^2 \eta.
\end{align*}
Let us observe that the variable $\xi$ parametrize the family of confocal ellipses centred at $c_1$  and $c_2$. Indeed, a direct computation shows
\begin{equation}
\label{eq:def_function_ellipses}
f(x)= \vert x-c_1\vert+\vert x-c_2\vert-2 = 2 (\cosh \xi-1)  .
\end{equation}
The Hamiltonian for the $\xi$ coordinate has a unique hyperbolic critical point corresponding to $\xi =0$, i.e. the collision-reflection solution bouncing between $c_1$ and $c_2$, spanning the segment connecting the two centres.
The phase portrait of the $\xi$ Hamiltonian is depicted in Figure \ref{fig:phase_portrait1} in the appendix for some choice of masses.

Note that any branch of separatrix uniquely determines a family of solutions that are asymptotic to the collision-reflection solution. The value $I_0$ of the first integral for the $\eta$ coordinate is then completely determined by this choice and the resulting solutions are monotone, as shown in the appendix.

\begin{figure}
	\label{fig:spirals}
	\centering
	\includegraphics[width=0.85\textwidth]{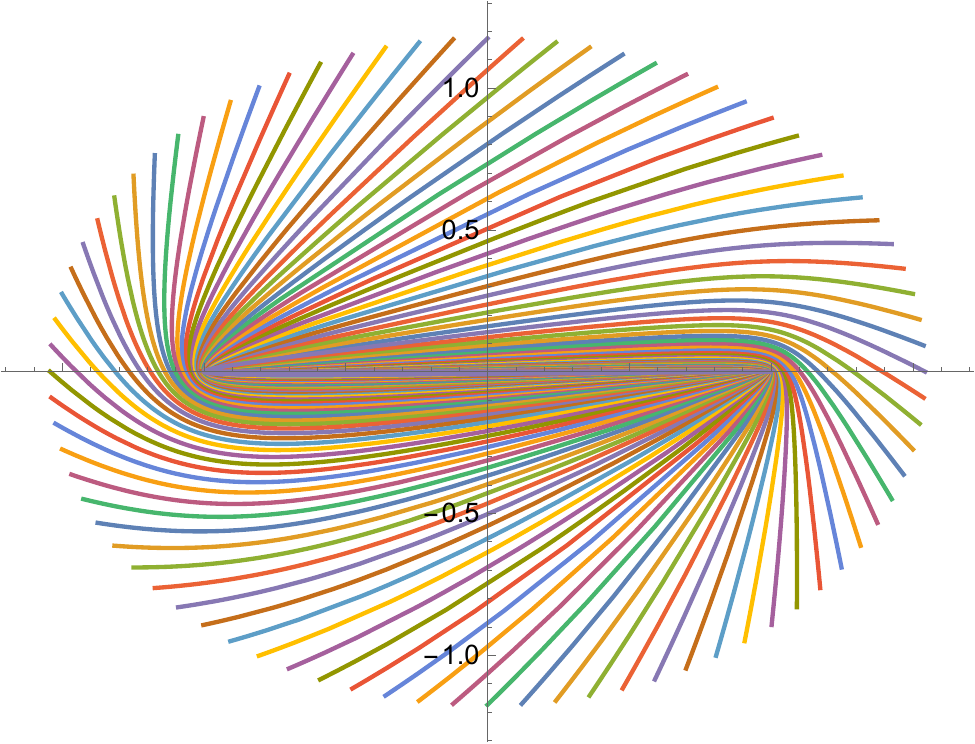}
	\caption{The curves $s_{\pm\infty}(t,\theta_0)$ given in equation \eqref{eq:def_spirals}. The mass parameters are $m_1 = \frac{3}{2}$,$m_2 = \frac{1}{2}$ and the energy is $h = 2$.}
\end{figure}

We have four families of solutions, up to time inversion, depending on whether the $\xi$ and $\eta$ are increasing or decreasing. Without loss of generality we pick the solution $\xi_{I_0}(s)$,  $s \in (-T_0,0)$ which is decreasing and denote as $\eta_{I_0}^\pm(s,\theta_0)$ the increasing/decreasing solution having $\eta^\pm_{I_0}(-T_{0}) = \theta_0$ for\footnote{Note that if $\eta$ is a solution, so is $\eta+ 2\pi$} $\theta_0 \in \mathbb{S}^1$. We consider the two families
\begin{equation*}
	\tilde{s}_{\pm\infty}(s,\theta_0) = \Phi (\xi_{I_0}(s),\eta^\pm_{I_0}(s,\theta_0)), \quad \theta_0 \in \mathbb{S}^1, s\in (-T_0,0).
\end{equation*}
After a time reparametrization $s(t,\theta_0)$ that depends on the trajectory and reverts these curves to solutions of \eqref{eq:motion_equation}, we obtain the family of solutions 
\begin{equation}
	\label{eq:def_spirals}
	s_{\pm}(t,\theta_0)  = \tilde{s}_{\pm\infty}(s(t,\theta_0),\theta_0), \quad t \in \mathbb{R}
\end{equation}
that have the following asymptotic properties
\[
\lim_{t\to-\infty} \frac{s_{\pm\infty}(t,\theta_0)}{\vert s_{\pm\infty}(t,\theta_0)\vert} =  e^{i\theta_0}, \quad \lim_{t\to +\infty} f( s_{\pm\infty}(t,\theta_0)) = 0, \quad \theta\in \mathbb{S}^1.
\]
The image of the $s_{\pm \infty}$ in the original coordinate system, for some specific value of the masses, is depicted in Figure \ref{fig:spirals}. An explicit parametrization is given in the Appendix \ref{appendix:2_centres} (equations \eqref{eq:parametrization_xi} and \eqref{eq:parametrization_eta}).

\subsection{Convergence of the stable/unstable manifolds}
The main observation of this section is that the curves $\{\gamma_k^{x,y}\}_{k\in \mathbb{Z}}$ in Definition \eqref{def:minimizers} converge, in a suitable sense, to the family $s_{\pm \infty}(t,\theta_0)$ defined in \eqref{eq:def_spirals} as $k \to \pm \infty$. Preliminarily, we prove the following consequence of the negativity of the curvature of the Jacobi-Maupertuis metric.

We say that a geodesics $\eta :[0,\infty)\to \mathbb{R}^2\setminus[c_1,c_2]$  with $\eta(0) \notin [c_1,c_2]$ is asymptotic to the collision-ejection solution if $\lim_{t\to +\infty} d(\eta(t),[c_1,c_2]) = 0$. Equivalently we can formulate this condition in terms of the function $f$ in \eqref{eq:def_function_ellipses} and require that $\lim_{t\to +\infty} f(\eta(t)) = 0$.

\begin{lemma}
	\label{lemma:uniqueness_spirals}
	For any $x \in \elld$ there is a unique $(t_0,\theta_0) \in (-T_0,0)\times \mathbb{S}^1$ such that $s_{\pm\infty}(t_0,\theta_0) = x$. Similarly, there is a unique $(t_0,\theta_0) \in (-T_0,0)\times \mathbb{S}^1$ such that $\frac{d}{dt}\vert_{t =t_0} s_{\pm\infty}(t,\theta_0)$ is orthogonal to $\elld$(and pointing towards the centres). Moreover, these are the only geodesics asymptotic to the collision-reflection solution satisfying the respective boundary conditions.
\begin{proof}
   A closer inspection of the phase portrait depicted in Figure \ref{fig:phase_portrait1} for the Hamiltonian of $\xi$ shows that $\xi$ can be either monotone or have a unique critical point. This implies that all solutions $\eta(t)$ for which the function $f$ defined in \ref{eq:def_function_ellipses} satisfies
   \[
   \lim_{t \to +\infty} f(\eta(t)) =0
   \]
   belong to the family $s_{\pm\infty}(t,\theta_0)$. It remains to show the uniqueness part. Assume that $\eta_1$ and $\eta_2$ are two geodesics of the family having $\eta_1(0) = \eta_2(0)=x$.  By the negativity of the curvature they do not have any other intersection point. Let us consider the following geodesic triangles. Pick the line $\ell'$ orthogonal to $[c_1,c_2]$ through its midpoint and define an increasing sequence $T_1^n>0$ requiring that $\eta_1(T_1^n)\in \ell'$. Consider the geodesic $\xi$ orthogonal to $\eta_1$, passing through $\eta_1(T_1^n)$, and define $\eta_2(T_2^n)$  to be the intersection of $\xi$  and $\eta_2$ determined by the condition that $\eta_1*\xi*\eta_2^{-1}$ is homotopically trivial.
	Clearly, as $n\to \infty$ the length of $\xi$ converges to zero since both $\eta_1$ and $\eta_2$ are asymptotic to $[c_1,c_2]$. 
	The angle between $\dot{\eta}_1(T_1^n)$ and $\dot{\eta}_2(T_2^n)$ converges to zero as well since $\eta_1$ and $\eta_2$ approach the collision ejection solution in the $C^1$ norm, at least away from the singularities $c_1$ and $c_2$. Thus, the angle $\alpha_n$ between $\dot{\xi}$ and $\dot \eta_2$ converges to  $\frac{\pi}{2}$. Let $\beta$ the angle between $\dot{\eta}_1(0)$ and $\dot{\eta}_2(0)$ and $\Omega_n$ the triangle bound by $\eta_1,\xi$ and $\eta_2$. Gauss-Bonnet theorem implies that
	\[
	    \int_{\Omega_n} \kappa_{JM}+ \alpha_n +\beta + \frac{\pi}{2} = 2 \pi. 
	\]
	However, taking the limit, this is possible if and only if $\eta_1 = \eta_2$.
	
	The proof of the second item is similar. Recall that the line $\elld$ is negatively curved as long as we choose the normal to point in the half space containing the centre (see Lemma \ref{lemma:convexity_mirror}). Assume that there are two geodesics $\eta_1$ and $\eta_2$ in the family $s_{\pm\infty}(t,\theta_0)$ both orthogonal to $\elld$. Without loss of generality we may assume $\eta_1(0)>\eta_2(0)$. We can build geodesic parallelograms $Q_n$ having the sum of the turning angles arbitrarily close to $2 \pi$, as showed before. Thus for some $\ve_n \to 0$,
	\[
	\int_{Q_n} \kappa_{JM}+\int_{[\eta_1(0),\eta_2(0)]} \kappa_{\elld}+2\pi-\ve_n = 2 \pi. 
	\]
	However, this is not possible unless $\eta_1$ and $\eta_2$ coincide. 
\end{proof}
\end{lemma}

With a slight abuse of notation, we can thus denote by $s_{\pm\infty}^x$ and $s_{\pm\infty}$ the unique geodesics parametrized as solutions of \eqref{eq:motion_equation} on $[0,\infty)$ and having \[
s_{\pm\infty}^x(0) = x, \quad \dot s_{\pm\infty}(0) \perp \elld.
\]

Let us parametrize the $\gamma_k^{x,y}$ so that they are solutions of \eqref{eq:motion_equation} and they are defined on some interval $[0,T_k^{x,y}]$. Then we extend each $\gamma_k$ on $[0,\infty)$ setting $\gamma_k^{x,y}(t) = \gamma_k^{x,y}(T_k^{x,y})$ for $t\ge T_k^{x,y}$. With a slight abuse of notation we call these extensions $\gamma_k^{x,y}$ as well.
We have the following.
\begin{prop}
	\label{prop:existence_eteroclinics}
     For any compact interval $I\subset \elld$, $x,y\in I$ and $T>0$ the maps $(x,t)\mapsto \gamma_k^{x,y}$ converge in the $C^2(I \times [0,T])$ topology to $s_{\pm\infty}^x$, uniformly in $y$ as $k \to \pm \infty$.
 \begin{proof}
 	We consider the case $k \to +\infty$. The other case is completely analogous. The steps of the proof are the following. First, we show that, for fixed $x,y \in I$, $\{\gamma_k^{x,y}\}_k$ admits a convergent subsequence  to a solution of \eqref{eq:motion_equation} on any compact interval $[0,T]$ for all $T>0$. Then, we show that this solution must be asymptotic to the collision-ejection one and thus is the $s^x_{\infty}$ given in Lemma \ref{lemma:uniqueness_spirals}. Since the same argument applies to any subsequence $\{\gamma_{k_n}^{x,y}\}_{n\in \mathbb{N}}$ we can conclude that the whole $\{\gamma_k^{x,y}\}_k$ converges to $s^{x}_{\infty}$.
 	Next, we show that the family $\{\gamma_k^{x,y}\}_{k,y}$ (seen as functions of the variables $(x,t)$) is compact in the $C^2(I\times [0,T])$ topology as well, for any $T>0$. Since the point-wise limit is necessarily the map $(x,t) \mapsto s^x(t)$, we get again convergence of the full family $\{\gamma_k^{x,y}\}_{k,y}$ to $s^x_{\infty}$ and the proposition follows.
 	
 	The first observation we make is that, for any $x,y\in I$ and $k\in\mathbb{Z}$, there exists $C_T>0$ such that 
 	\[
 	     \int_0^T\vert\dot{\gamma}^{x,y}_{k}\vert^2\le C_T.
 	\]
 	Indeed, let us consider the lifts $\tilde \gamma_k^{x,y}$ of the $\gamma_k^{x,y}$ to the universal cover of $\mathbb{R}^2\setminus [c_1,c_2]$. Since all these curves satisfy the energy equation
 	\[
 	h = \frac12\vert \dot{\gamma}\vert^2-V(\gamma),
 	\]
 	their velocity is bounded outside small balls centred at $c_1$ and $c_2$. This means that the image of the  ${\gamma}_k^{x,y}$ is bounded in the universal cover. Since the $L^2$ norm of the velocity coincides with the JM length for any curve satisfying the energy equation, we have
 		\[
 	\int_0^T\vert\dot{\gamma}^{x,y}_{k}\vert^2 = \int_0^T \sqrt{2(h+V)}\vert \dot{\gamma}_k^{x,y}\vert\le C_T.
 	\] 
 	
 	In particular, it follows that the family $\{\gamma_k^{x,y}\}_k$ is weakly precompact in $H^1[0,T]$ for any $T>0$. Passing to a subsequence, we can assume that $\gamma_k^{x,y}$ converges uniformly to a continuous curve $\gamma_\infty^{x,y}$ on any $[0,T]$. 
 	
 	Next, we show that $d(\gamma_\infty^{x,y}\vert_{[0,T]},\{c_1,c_2\})\ge \delta_T$ for some $\delta_T>0$.
 	Let us assume by contradiction that this is false and define $\bar{t} = \inf\{t :d(\gamma_\infty^{x,y}(t),\{c_1,c_2\})=0\}\in (0,T]$. The standard bootstrap argument shows that $\gamma_\infty^{x,y}\vert_{[0,t]}$ is a solution of \eqref{eq:motion_equation} for any $t< \bar{t}$.  	
    Let us pick $r>0$ small and consider the ball $B_1 = B_r(c_1)$. Restricting the $\gamma_{k}^{x,y}$, we obtain a family of length minimizers joining points in $\partial B_1$ and approaching a collision. By the blow-up argument in \cite[Lemma 4.6]{baranzini2024chaotic} or \cite[Section 4.4]{SoaTer2012} we conclude that a suitable rescaling of this minimizers converge to a limit having an angular variation of at least $2\pi$. Since, thanks to minimality, none of the $\gamma_k^{x,y}$ can cross the segment $[c_1,c_2]$ which is the support of a solution, we conclude that $\gamma_\infty^{x,y}\vert_{[0,\bar t)}$ must coincide with a piece of the collision-ejection solution. This however is not possible since $\gamma_\infty^{x,y}(0) = x$.
    Having proved that there exists $\delta_T>0$ as above, an application of Ascoli-Arzelà theorem yields the $C^2[0,T]$ convergence of a subsequence for all $T>0$ to a collision-less curve $\gamma_{\infty}^{x,y}$.   
       
    Now, let us prove that $\gamma^{x,y}_{\infty}$ is asymptotic to the collision-reflection solution. This follows from the minimality property of the collision-ejection solution: it is the only positive energy closed minimizer of the Jacobi Maupertuis length. We now show that the curves $\gamma_k^{x,y}$ must approach it as $k\to \infty$.
    Indeed, let us assume by contradiction that there exists $\delta>0$ such that $d(\gamma_k^{x,y},[c_1,c_2])\ge \delta$ for all $k$ big enough. For $k$ large enough, the curve built concatenating a segment joining $x$ to $c_1$ followed by $k$ iteration of the collision-ejection solution  and a segment joining $c_1$ to $y$ is eventually shorter than $\gamma_k^{x,y}$.
    Thus, $\lim_k d(\gamma_k^{x,y},[c_1,c_2]) =0$.   
    Recall that the function $f$ defined in \eqref{eq:def_function_ellipses} is either monotone or has a unique minimum point along solutions of the two centre problem (compare with \eqref{eq:function_ellipse}). The restriction of $f$ to the $\gamma_k^{x,y}$ has then a unique minimum since the $\gamma_k^{x,y}$ approach $[c_1,c_2]$ as $k \to +\infty$. This must be obtained on a sequence $T_k\to \infty$ since, otherwise, we could find a subsequence of the $\gamma_k^{x,y}$ converging to a collision solution. Thus, being the limit of decreasing functions on each fixed interval $[0,T]$, $f(\gamma_\infty^{x,y})$ is monotone decreasing and $\gamma_{\infty}^{x,y}$ is asymptotic to the collision-reflection solution. 
   Thanks to Lemma \ref{lemma:uniqueness_spirals}, we know that the unique geodesic passing through $x$ and asymptotic to $[c_1,c_2]$ is $s_{\infty}^x$. 
   
   We have thus proved that any subsequence of the $\{\gamma_{k}^{x,y}\}_{k\in \mathbb
   N}$ which converges uniformly on compacts must converge to the same limit $s^x_\infty$ which in particular is independent on $y$. Thus the full sequence $\{\gamma_k^{x,y}\}_k$ converges to $s^x_{\infty}$.
   
   Now, we prove the uniform convergence on $I\times [0,T]$ for every $T>0$, uniformly in $y \in I$.  By the previous point we know that, if a subsequence converges uniformly, the limit map is necessarily given by   $(x,t) \mapsto s^x_{\infty}(t)$. Thus, we need only to show the compactness of the family $\{\gamma_k^{x,y}\}_{k \in \mathbb{N}, y \in I}$, seen as functions of the variables $(x,t)$.  
   To do so, we want to apply again the Ascoli-Arzelà theorem. 
   To this extent we have to compute the derivatives of the $\gamma_k^{x,y}$ with respect to $x$. Let us observe that 
   \[
   J_{x_0}(t)=\left(\frac{d}{dx} \gamma_k^{x,y}(t)\right)\vert_{x=x_0}
   \]
   is a Jacobi field along $\gamma_k^{x_0,y}$. Since we have that the curvature $\kappa_{JM}$ of the Jacobi-Maupertuis metric satisfies $-c\le \kappa_{JM}<0$, the growth of the components of $J_{x_0}(t)$ in a parallel frame is at most exponential. We have thus to show that the (euclidean) norm of a orthonormal parallel frame is bounded. Clearly, such a frame is given by $\dot{\gamma}_k^{x,y}$ and $\left(\dot{\gamma}_k^{x,y}\right)^\perp$.  Thanks to the energy equation, to control the norm of the frame, we have to show that the geodesics do not pass arbitrarily close to collision on any finite interval $[0,T]$.
   We argue again by contradiction and assume that there is a subsequence $\{k_n\}_{n \in \mathbb{N}}$ for which 
   \begin{equation}
   	\label{eq:collision}
       \inf_{t \in [0,T]} d(\gamma_{k_n}^{x_n,y_n}(t),\{c_1,c_2\}) \to 0
   \end{equation}
   as $n \to \infty$. By compactness, we can refine the subsequence and assume that $x_n\to \bar{x}$ monotonically.
   We now employ the following form of monotonicity. No minimizer $\gamma_{k}^{x,y}$ can intersect $\gamma_{k}^{x, \inf I}$ nor $\gamma_{k}^{x,\sup I}$ for fixed $k$ and $x \in I$.
   In particular, if \eqref{eq:collision} holds, $\gamma_{k_n}^{x_n,\sup I}$ must satisfy 
    \[
   \inf_{t \in [0,T]} d(\gamma_{k_n}^{x_n,\sup I}(t),\{c_1,c_2\}) \to 0
   \]
   as well.  If $x_n$ is decreasing we obtain again by monotonicity that
     \[
   \inf_{t \in [0,T]} d(\gamma_{k_n}^{\bar{x},\sup I}(t),\{c_1,c_2\}) \to 0
   \]
   which contradicts the convergence for fixed $x$ and $y$ already proved. Thus $x_n$ must be increasing. However, in this case, we have that for any $z < \bar{x}$
     \[
   \inf_{t \in [0,T]} d(\gamma_{k_n}^{z,\sup I}(t),\{c_1,c_2\}) \to 0
   \]
  which is again a contradiction.
   
  Thus, any subsequence converging uniformly on $I \times [0,T]$ converges to a immersed surface contained in $\mathbb{R}^2 \setminus [c_1,c_2]$. By a standard bootstrap argument, it converges in the $C^2([0,T]\times I)$ topology as well. Since the only possible point-wise limit is the map $(x,t)\to s_\infty^x(t)$, we conclude that $\gamma_k^{x,y}\to s_\infty^x(t)$ in the $C^2$ topology on $I \times [0,T]$ for any $T>0$, uniformly in $y\in I$.   
 \end{proof}
\end{prop}

\subsection{Proof of Theorem \ref{thm:main_2_centres}}

\begin{thm}
	\label{thm:periodic_two centers}
	Let $\elld$ an admissible line. There exists $k_0\in \mathbb{N}$ such that, for any $r\ge0$, the  length minimizers in the space $\mathcal{C}_{i_1,\dots,i_{2r}}$ given in \eqref{eq:def_space_paths} are billiard trajectories provided that $\vert i_j\vert>k_0$ and $\mathrm{sign}(i_j) = - \mathrm{sing}(i_{j+1})$ for all $j$.
    \begin{proof}
    	We denote by $x_\pm = s_{\pm\infty}(0)$ the initial point of the geodesics orthogonal to $\elld$. First we prove that there exist small neighbourhoods of $x_\pm$ for which all the $s_{\pm \infty}^x$ are completely contained in the half-plane $\Pi_{\elld}$.
    	
    	Let us consider the universal cover of $\mathbb{R}^2\setminus [c_1,c_2]$, a lift of $s_{+\infty}$ starting from $\elld^0$, a lift of $\elld$, and all the successive lifts $\elld^m$ of $\elld$ for $m>0$. Thanks to the admissibility condition, the distance between the lift of $s_{\infty}$ and any $\elld^m$ is $d_m>0$ (and increasing).     	
    	Let us further define
    	\[
    	f_0 = \min_{x \in \elld} f(x),
    	\]
    	where $f$ is the function defined in $\eqref{eq:def_function_ellipses}$. Since $s_\infty$ is asymptotic to the collision ejection solution, there exists $T_0>0$ such that 
    	\[
    	\frac{f_0}{4} = f(s_\infty(T_0))>f(s_{\infty}(t)), \, \forall t>T_0.
    	\]
    	
    	Thanks to the continuous dependence on initial conditions, for any $\ve>0$ there exists a $\delta(\ve)>0$ such that
       	\[
       	 \forall t \in [0,T_0], \,\forall x :  \vert x- x_+\vert <\delta(\ve), \quad \vert s_{+\infty}^x(t)-s_{+\infty}(t)\vert <\ve.
       	\] 
       	Since $f$ is continuous there exists $\delta_1>0$ such that  
       	\[
       	\forall t \in [T_0,\infty), \, \forall x:   \vert x- x_+\vert <\delta_1, \quad f(s^x_{\infty}(t))\le \frac{f_0}{2}.
       	\]
       	If we choose $\ve = \frac{d_1}{2}$ and pick $2\delta' = \min \{\delta(\ve), \delta_1\}$ we have that for all $x$ in $I_+ = [x_+-\delta',x_++\delta']$ all solution $s_{\infty}^x(t)$ never touch $\elld$ and for all times $t>T_0$ are properly contained in an ellipse of diameter $\frac{f_0}{2}$.
       	Performing the same construction for $s_{-\infty}$, we can find an interval $I_-$ of $x_-$ with the same properties.
       	
       	Thanks to Proposition \ref{prop:existence_eteroclinics}, we know that the family $\{\gamma_k^{x,y}(t)\}_{k\in \mathbb{Z}}$ converge uniformly on $I\times [0,T]$ to $s_{\pm\infty}^{x}(t)$, for any fixed $I$ and $T>0$. This means that there exists $k_+$ big enough so that
    \[
   \forall x\in I_+, \, \forall y \in I_-, \,\forall k\ge k_+, \quad
    f(\gamma_k^{x,y}(T_0))<f_0
    \]
    and the distance between $\gamma_k^{y,x}(t)$ and the first copy of $\elld$ is at least $\frac{d_0}{4}$ for all $t \in [0,T_0]$. Performing the same argument for $I_-$ we can find a common $k_0$ so that, for any $k :\vert k\vert >k_0$,  $x \in I_+$ and $y \in I_-$ all the $\gamma_k^{x,y}$  intersect $\elld$ only in $x$ and $y$. 
    
    Let us observe that the angle between $\dot{s}^x_{+\infty}(0)$ and $\elld$ is monotone in a neighbourhood of $x_+$. This can proved using the same argument as in Lemma \ref{lemma:uniqueness_spirals}. We denote this angle as $\alpha^x$. Up to shrinking $I_+ = [x_+-\delta',x_++\delta']$, there exists a $\theta_0$ such that $\vert \alpha^{x_+ \pm \delta'}-\frac{\pi}{2}\vert \ge 2 \theta_0$.
    
    Since the $\gamma_k^{x,y}$ converge uniformly in the $C^{2}$ topology, up to passing to a larger $k_0$, the angle $\alpha_k^{x_\pm\pm \ve,y}$ made by  $\dot{\gamma}^{x_\pm\pm \ve,y}_k(0)$ satisfies $\vert \alpha_k^{x_\pm\pm \ve,y}-\frac{\pi}{2}\vert>\delta$ for any $k$  such that $\vert k \vert\ge k_0$ and $y \in I_-$. An analogous observation holds for $\dot{s}_{-\infty}^x(0)$ as well. 
    
    We now prove the existence of periodic trajectories. Assume that $(i_1,\dots i_{2r})$ is a sequence such that  $\mathrm{sign}(i_j) = - \mathrm{sing}(i_{j+1})$ for all $j$. Without loss of generality assume $i_1>0$. Let us consider the space $\mathcal{C}^\ve_{(i_1,\dots,i_{2r})}$ defined as
    \[
        \mathcal{C}^\ve_{(i_1,\dots,i_{2r})} = \{ 
        \gamma_{i_1}^{x_1,x_2}*\dots*\gamma_{i_{2r}}^{x_{2r},x_1}:\, x_{i_{2k-1}}\in I_+,\, x_{i_{2k}}\in I_-,\, 1 \le k \le r
        \}
    \]
    corresponding to a small $2r$-dimensional cube. Since the space is compact, length-minimizers in $\mathcal{C}^\ve_{(i_1,\dots,i_{2r})}$ always exists. Moreover, by construction, all curves in $\mathcal{C}^\ve_{(i_1,\dots,i_{2r})}$ lie in the half-plane containing the centres. 
    
    If $\gamma$ is a length minimizer and $x$ a bouncing point belonging to the interior of $I_+$ or $I_-$, one can show that the elastic reflection condition holds exactly as in Theorem \ref{thm:existence_periodic}. To conclude, we prove that minimizers never touch the boundary of $\mathcal{C}^\ve_{(i_1,\dots,i_{2r})}$.

    Suppose by contradiction that a bouncing point occurs on the boundary. Without loss of generality we can assume that this bouncing point is $x_+ +\delta'$ and consider the portion of $\gamma$ given by $\gamma_{-m_1}^{y_1,x_+ +\delta'}*\gamma_{m_2}^{x_+ +\delta', y_2}$ for some $y_i\in I_-$ and $m_i\ge0$.     
    By construction, both curves $\gamma_{m_i}^{x_++\delta',y_i}$ make an angle $\beta_i$ which satisfies $\vert\frac{\pi}{2}-\beta_i\vert \ge \theta_0$. This implies that both of these segments are not critical points of the length among paths joining $y_i$ to $\elld$ and making $m_i$ turns around $[c_1,c_2]$. Shifting the endpoint $x_++\delta'$ towards $x_+$ is a length-decreasing variation for both of them and thus $\gamma$ cannot be a length-minimizer, a contradiction.
\end{proof}
\end{thm}

 Now we prove a statement analogous to Theorem \ref{thm:non_periodic_orbits}. Recall that to any sequence $\{i_j\}_{j \in \mathbb{Z}}$ we can associate a billiard trajectory obtained concatenating arcs winding $i_j-$times around the centres. Let $\xi:\mathbb{R}\to \mathbb{R}^2\setminus \{c_1,c_2\}$ a billiard trajectory, assume that $\xi(0) \in \elld$ and denote by $\xi_i, i \in \mathbb{Z}$ the restriction of $\xi$ between two consecutive bounces.
\begin{thm}
	\label{thm:non_periodic_orbits_2centres}
	Let $\elld$ an admissible line. There exists $k_0\in \mathbb{N}$ such that, for any sequence $(i_j)_{j \in \mathbb
	Z}$ such that  $\vert i_j\vert>k_0$ and $\mathrm{sign}(i_j) = - \mathrm{sing}(i_{j+1})$ for all $j$, there exists a billiard trajectory $\xi$ such that $\xi_j \in \mathcal{C}_{i_j}$ for all $j$.
\begin{proof}
	The proof of the Theorem is similar to the one of \ref{thm:non_periodic_orbits}. The only significant difference is that we cannot bound the distance from the centres $c_1,c_2$ in a uniform way, for all the trajectories we are going to build. We argue as follows.
	Take a sequence $(i_j)_{j \in \mathbb{Z}}$ with $\vert i_j \vert \ge k_0$ (the one given in Theorem \ref{thm:periodic_two centers}) and $ \mathrm{sign}(i_j) = - \mathrm{sing}(i_{j+1})$ for all $j$. For any $n \in \mathbb{N}$, let us consider the finite subsequence $(i_{-n}, \dots, i_{n-1})$ and the relative periodic billiard trajectories of Theorem \ref{thm:periodic_two centers} which we call $\eta_n$. Without loss of generality we can assume that $\eta_n$ are extended to the whole $\mathbb{R}$ and $\eta_n(0)\in \elld$. 
	
	Let $T>0$. Clearly, any $\eta_n$ has finitely many bounces on $[-T,T]$. This implies that, for any such interval there are just a finite number of indexes $k_{1}, \dots k_m$ for which the segments of $\eta_n$ between consecutive bounces belong to $\mathcal{C}_{k_i}$. By the same continuity argument of Theorem \ref{thm:non_periodic_orbits}, the family $\{\eta_n\vert_{[-T,T]}\}_n$ is pre-compact in the $C^0([-T,T])$ topology. By the usual diagonalization argument and the unfolding argument of Theorem \ref{thm:non_periodic_orbits}, this yields a subsequence converging on any compact interval $[-T,T]$ to a $C^1$ curve $\eta$, in the $C^1$ topology. Finally, by construction (recall that all trajectories built in \ref{thm:periodic_two centers} are almost orthogonal to $\elld$), all the $\eta_n$ are transversal to $\elld$ and thus $\eta^{-1}(\elld)$ is discrete. So, we have $C^2$ convergence to a solution of  \eqref{eq:motion_equation} any compact interval of $\mathbb{R}\setminus\eta^{-1}(\elld)$.
\end{proof}
\end{thm}

\section{The symbolic dynamic}
\label{sec:symbolic}
In this section, we briefly describe how to build a fairly rich symbolic dynamics using the billiard orbits built in Theorem \ref{thm:non_periodic_orbits} and Theorem \ref{thm:non_periodic_orbits_2centres}. Let $\ell_{[c_1,c_2]}$ be the line passing through the centres and let us consider $\sigma$, a length minimizer joining $c_1$ with a point in $\ell_{[c_1,c_2]}$, sufficiently far off, so that every minimizer intersects it only in its interior. Let us define two surfaces
\begin{align*}
    \Sigma_+ &= \{
    v \in T_{\sigma(s)}\left(\mathbb{R}^2\setminus\{c_1\}\right): \langle v ,\dot{\sigma}(s) \rangle> 0
    \} \\ \Sigma_- &= \{
    v \in T_{\sigma(s)}\left(\mathbb{R}^2\setminus\{c_1\}\right): \langle v ,\dot{\sigma}(s) \rangle< 0
    \}.
\end{align*}
Clearly, $\Sigma_+\cap \Sigma_-= \emptyset$  and $\Sigma = \Sigma_+\cup \Sigma_-$ coincides with the two connected component of open cylinder given by $H^{-1}(h) \cap T \left(\mathbb R^2\setminus\{c_1\}\right)$ determined by the lift of $\sigma$ and its inverse.
Let us denote by $\Phi_t: \Sigma \to \mathbb{R}^2$ the restriction of the billiard \emph{flow}\footnote{Meaning that we follow billiard trajectories, parametrized as piecewise solutions of \eqref{eq:motion_equation}.} to $\Sigma$, at time $t\ge0$. 
For every $v \in \Sigma$, let us define:
\[
T^\pm(v) = \inf \{ \pm t> 0 : \Phi_{t}(v) \in \Sigma \}, \quad U^\pm = \{v \in \Sigma : T^\pm(v)<\infty\}.
\]

Let $U^{(0)} = U^+\cap U^-$ and, for any $n> 0$, let us define $ U^{(n)}= \{ \Phi_{T^+(v)}(v) : v \in U^{(n-1)}\} \cap U^+ $. Similarly set $U^{(-n)} = \{ \Phi_{T^-(v)}(v) : v \in U^{(-n+1)}\} \cap U^-$. Let us define the set of non-escaping points
\[
    \mathcal
    U = \bigcap_{n= -\infty}^\infty U^{(n)}.
\]
For any $v \in \mathcal{U}$ we can define an itinerary, i.e. a sequence $\omega(v) =(\omega_n)_{n \in \mathbb Z} \in \{0,1\}^\mathbb{Z}$ as follows. We set $v_0 = v$ and $\omega_0=0$ if $v_0 \in \Sigma^+$ or $\omega_0 =1$ otherwise. Then set
\[
v_n =\begin{cases}
	\Phi_{T^+(v_{n-1})}(v_{n-1}) &\text{ if }n>0\\
	\Phi_{T^{-}(v_{n+1})}(v_{n+1}) &\text{ if } n<0
\end{cases} \text{ and } \omega_n = \begin{cases}
	0 &\text{ if } v_n \in \Sigma^+\\\
		1 &\text{ if } v_n \in \Sigma^-.
\end{cases}
\]
We thus obtain a map
\[
\omega: \mathcal{U} \to \{0,1\}^\mathbb{Z}.
\]
Thanks to Theorem \ref{thm:non_periodic_orbits} and \ref{thm:non_periodic_orbits_2centres} the set $\mathcal{U}$ is always non-empty. Theorem \ref{thm:non_periodic_orbits} ensures that to any sequence in $\{0,1\}^\mathbb
Z$ corresponds at least one minimizer. On the other hand, the set $\mathcal{U}$ is much larger. For instance, the constant sequence of $0$ can be realized by infinitely many distinct trajectories. It is the same for all minimizers in $\mathcal{C}_k$, $k\ge0$.
Similarly, Theorem \ref{thm:non_periodic_orbits_2centres} gives a correspondence between a set of minimizers and sequences in $\{0,1\}^\mathbb{Z}$ having at least $k_0$ repetitions. 
\begin{prop}
	The map $v \mapsto \omega(v)$ is a semi-conjugation with a sub-shift of finite type.
	\begin{proof}
		We give a brief sketch of the proof. The reader is referred to \cite[Theorem 6.6]{baranzini2024chaotic}, \cite[Theorem 1.7]{BarCanTer2021} or \cite[Section 2.5]{Barutello_2023} and references therein for more details.
		
		The space $\{0,1\}^\mathbb{Z}$ is endowed with the topology of cylinder sets and it is compact. Continuity of the map $\omega:\mathcal U \to \{0,1\}^{\mathbb{Z}}$ is a straightforward consequence of continuous dependence of initial conditions. 
		
		In the case of Theorem \ref{thm:main_perturbative}, the map $\omega$ is surjective and there is nothing prove. In the case of Theorem \ref{thm:main_2_centres} the map is not surjective but the image is a sub-shift of finite type since we are just removing all finite sequences of length $\le k_0$ with different symbols.		
	\end{proof}
\end{prop}

\appendix
\section{Integrability of the 2 centre problem}
\label{appendix:2_centres}
In this section we collect some features of the $2-$centre problems for the reader convenience. Without loss of generality, we consider the Hamiltonian
\[
  H (p,x) = \frac{1}{2}\vert p \vert^2+\frac{m_1}{\vert x-e_1\vert}+\frac{m_2}{\vert x+e_1\vert},
\]

and assume $m_1\ge m_2>0$. We consider the following change of coordinates
\[
x = \Phi(\xi,\eta) = (\cosh \xi \cos \eta, \sinh \xi \sin \eta)
\]
A straightforward computation shows that
\[
\begin{cases}
	dx_1 = \sinh \xi \cos \eta \, d \xi -\cosh \xi \sin \eta \, d \eta,\\
	dx_2 = \cosh \xi \sin \eta \, d \xi +\sinh \xi \cos \eta \, d \eta.
\end{cases}
\]
and thus the euclidean metric in the new coordinates reads
\begin{align*}
   dx_1^2+dx_2^2& = (\sinh^2 \xi \cos^2 \eta+\cosh^2\xi\sin^2 \eta )(\, d \xi^2 + d\eta^2 )\\&=(\cosh^2\xi-\cos^2 \eta)(\, d \xi^2 + d\eta^2 )
\end{align*}
Moreover the functions $\vert x\pm e_1\vert$ have the following expressions
\begin{align*}
\vert x\pm e_1\vert &= \sqrt{\cosh^2 \xi\cos^2 \eta +1\pm2\cosh \xi\cos \eta+\sinh \xi^2 \sin^2 \eta}\\
&=\sqrt{\sinh^2 \xi+1 +\cos^2 \eta\pm2\cosh \xi\cos \eta } = \cosh \xi \pm \cos \eta
\end{align*}
and thus the $\xi$ coordinate is given by the following relation
\begin{equation}
	\label{eq:function_ellipse}
\cosh \xi  = \vert x+e_1\vert+\vert x-e_1\vert.
\end{equation}
Thus, the Hamiltonian in the new coordinates $\xi$ and $\eta$ reads
\begin{align*}
\tilde H (p_\xi,p_\eta,\xi,\eta) &=  \frac{p_\xi^2+p_\eta^2}{2(\cosh^2 \xi-\cos^2 \eta)}-\frac{m_1}{\cosh \xi+\cos \eta}-\frac{m_2}{\cosh \xi-\cos \eta}\\
& = \frac{\frac{p_\xi^2+p_\eta^2}{2}-m_1(\cosh \xi-\cos \eta)-m_2(\cosh \xi+\cos \eta)}{\cosh^2 \xi-\cos^2 \eta}.
\end{align*}
Thus, the flow of the Hamiltonian $\tilde H$ at energy $h>0$ is a time re-parametrization of the zero energy flow of the Hamiltonian $K$ given by
\[
 K(p_\xi,p_\eta,\xi,\eta) = \frac{p_\xi^2+p_\eta^2}{2}-m_1(\cosh \xi-\cos \eta)-m_2(\cosh \xi+\cos \eta)-h(\cosh^2\xi-\cos^2 \eta).
\] 
Now $K$ splits as a sum of two Hamiltonians $K_1$ and $K_2$ which depend solely on $\xi$ and $\eta$ respectively.
\begin{align*}
K_1(p_\xi,\xi)&= \frac12 p_\xi^2 - (\mu_1 +h \cosh \xi) \cosh \xi,\\
K_2(p_\eta,\eta) &= \frac12 p_\eta^2 + (\mu_2+h \cos \eta) \cos \eta.
\end{align*}
where $\mu_1 = m_1+m_2$ and $\mu_2 = m_1-m_2$. The first Hamiltonian has a unique critical point at the origin (see Figure \ref{fig:phase_portrait1}). A straightforward computation shows that this is a saddle point. For the second Hamiltonian we have to distinguish two cases.
\begin{figure}[h]
	\includegraphics[width=0.31\textwidth]{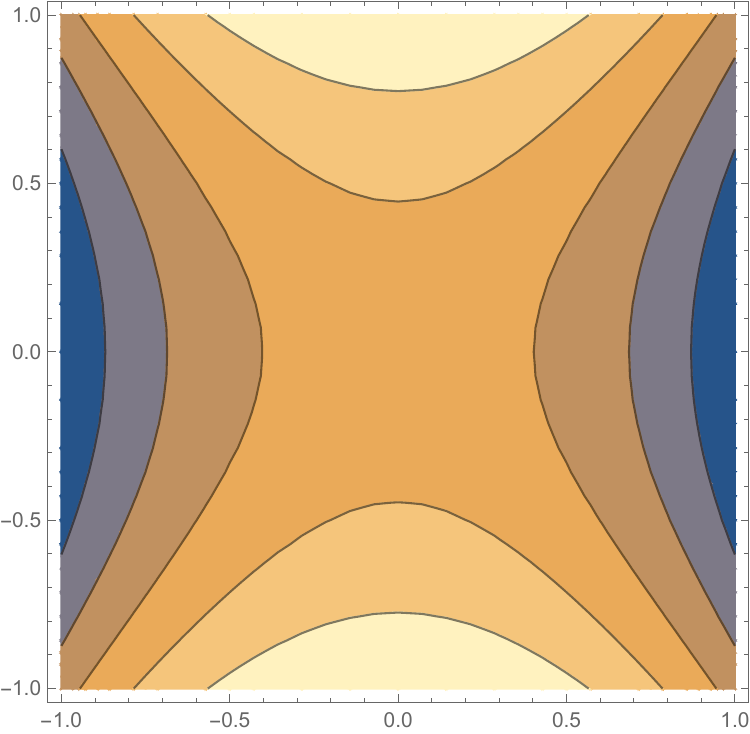}
	\includegraphics[width=0.31\textwidth]{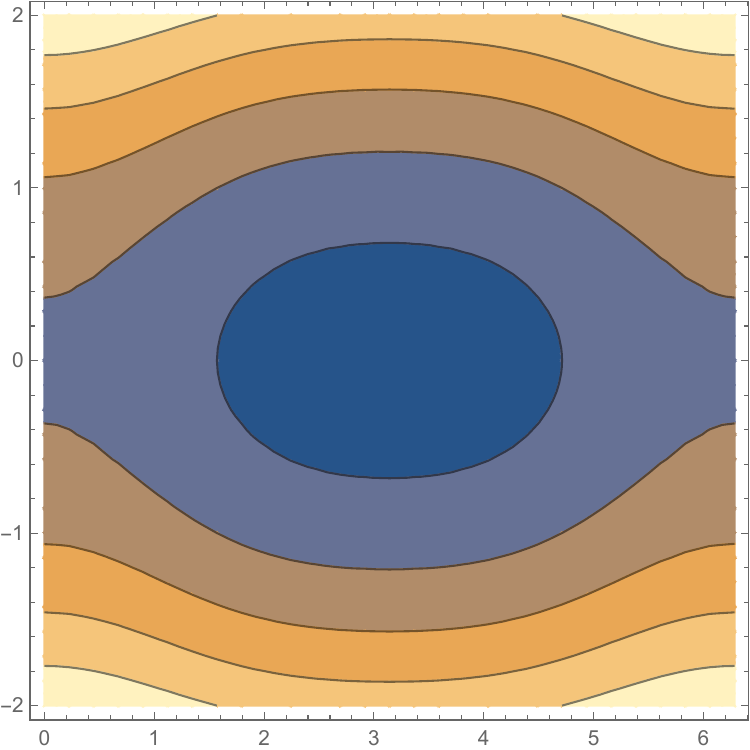}
	\includegraphics[width=0.31\textwidth]{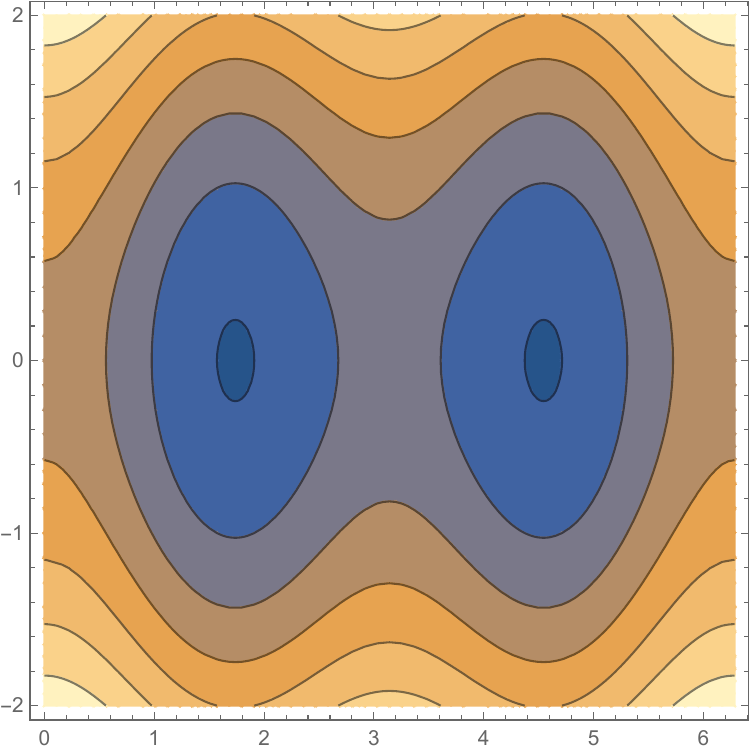}
	\caption{From left to right, the phase portrait of $K_1$, $K_2$ for  $h<\frac{\mu}{2}$ and $K_2$ for $h>\frac{\mu}{2}$.}
	\label{fig:phase_portrait1}
\end{figure}
\begin{itemize}
	\item \textbf{Case $\frac{\mu_2}{2}> h >0$.} In this case there a countably many critical points corresponding to $\eta_k = \pi k$, $k \in \mathbb{Z}$. For this interval of energies they are saddles if $k$ is even and minima if $k$ is odd.
	\item \textbf{Case $h>\frac{\mu_2}{2}$.} The critical points are given by $\eta_k = \pi k$, $k \in \mathbb{Z}$ and $\chi_k^\pm =\pm  \arccos\left(-\frac{\mu_2}{2h}\right)+2\pi k$, $k \in \mathbb{Z}$. They satisfy $\eta_{2k} < \chi_k^+ < \eta_{2k+1} < \chi_{k+1}^-$ for all $k \in \mathbb{Z}$.
	In this case, the $\eta_k$ are always saddles whereas the $\chi_k^\pm$ are always minima. One should note however that $\eta_{2k}$ and $\eta_{2k+1}$ are always at different energy levels, indeed $K_2(0,2k\pi) = \mu_2+h>K_2(0,(2k +1)\pi) = h-\mu_2$.
\end{itemize}
We are interested in the solutions $s^x_{\pm \infty}$ given in Lemma \ref{lemma:uniqueness_spirals}. Since the segment $[-e_1,e_1]$ is the minimum level of the function $f(x) = \vert x-e_1\vert+\vert x+e_1\vert$, which thanks to \eqref{eq:function_ellipse} is determined by the $\xi$ coordinate alone, we have to look for integral curves of $K_1$ having $\xi(t) \to 1$ as $ t \to \pm \infty$. The only possibility is to pick the separatrices, at energy $K_1 = -(\mu_1+h)$. This uniquely sets the energy for $K_2$ to be $\mu_1+h$. A straightforward computation shows that 
\[
   \frac{p_\eta^2}{2} = \mu_1+h-(\mu_2+h \cos \eta) \cos \eta= \mu_1-\mu_2\cos \eta+h(1-\cos^2 \eta) \ge 2 m_2
\]
and thus the $\xi(t)$ and $\eta(t)$ are both monotone. Without loss of generality we consider the case of $\eta(t)$ increasing.
Let us produce an explicit parametrization of the solutions $s_{\pm \infty}^x$ given in Lemma \ref{lemma:uniqueness_spirals}. The usual separation of variables gives the following expression for the solution having initial conditions $(\xi_0,\eta_0)$
\begin{align*}
\sqrt{2 h } t&=\int_{0}^{t}\frac{\dot{\xi}(t) dt}{\sqrt{-\bar{\mu}_1-1+\bar{\mu}_1\cosh \xi(t)+\cosh^2\xi(t)}}\\ &=\int_{\xi_0}^{\xi(t)}\frac{d \xi }{\sqrt{-\bar{\mu}_1-1+\bar{\mu}_1\cosh \xi+\cosh^2\xi}}  ,\\
\sqrt{2 h } t&= \int_{0}^{t}\frac{\dot{\eta}(t) dt}{\sqrt{\bar{\mu}_1+1-\bar{\mu}_2\cosh \eta(t)-\cosh^2\eta(t)}}\\ &= \int_{\eta_0}^{\eta(t)}\frac{d \eta}{\sqrt{\bar{\mu}_1+1-\bar{\mu}_2\cosh \eta-\cosh^2\eta}},
\end{align*}
where $\bar{\mu}_i = \frac{\mu_i}{h}$. The first integral can be reduced to
\[
\sqrt{ 2 h } t = \sqrt{\frac{2}{2+\bar{\mu}_1}}\mathrm{arctanh}\sqrt{\frac{(2+\bar\mu_1)(1+\cosh \xi)}{2(1+\bar\mu_1 +\cosh \xi)}} \Big \vert_{\xi_0}^{\xi(t)}, 
\]
yielding the following formula for $\xi(t,\xi_0)$
\begin{align}
	\label{eq:parametrization_xi}
	\xi(t,\xi_0) &= \mathrm{arcsech} \left(\frac{\bar\mu_1\cosh^2 \sigma(t,\xi_0)+2 }{2+\bar{\mu}_1}\right), \\
	\nonumber
\sigma(t,\xi_0) &= \sqrt{(2+\bar\mu_1)h } t-\mathrm{arctanh}\sqrt{\frac{(2+\bar\mu_1)(1+\cosh \xi_0)}{2(1+\bar\mu_1 +\cosh \xi_0)}}.
\end{align}
 The second integral can be explicitly computed as well, giving
 \begin{align*}
  \sqrt{2 h } t &= \frac{2}{\sqrt{2+\bar{\mu}_1+\sqrt{\Delta}}}F\left(\arctan\frac{\tan\frac{\eta}{2}}{\sqrt{2+\bar \mu_1-\sqrt{\Delta}}},\sqrt{\frac{2\sqrt{\Delta}}{2+\bar\mu_1+\sqrt{\Delta}}}\right)\Big\vert_{\eta_0}^{\eta(t)},\\
  \Delta &= 4+4\bar{\mu}_1+\bar{\mu}_2^2.
 \end{align*}
 Here $F(\varphi,k)$ stands for the Elliptic function of first kind of modulus $k$. This yields the following expression for $\eta(t,\eta_0)$
 \begin{align}
 	\label{eq:parametrization_eta}
 \eta(t) &= 2 \arctan \left(\sqrt{2+\bar\mu_1-\sqrt{\Delta}} \, \mathrm{tn}\left(\omega(t,\eta_0),\sqrt{\frac{2\sqrt{\Delta}}{2+\bar\mu_1+\sqrt{\Delta}}}
 \right)\right),\\
 \nonumber\omega(t,\eta_0) &= \sqrt{\frac{(2+\bar{\mu}_1+\sqrt{\Delta})h}{2}} t-F\left(\arctan\left(\frac{\tan\frac{\eta_0}{2}}{\sqrt{2+\bar{\mu}_1-\sqrt{\Delta}}}\right),\sqrt{\frac{2\sqrt{\Delta}}{2+\bar\mu_1+\sqrt{\Delta}}}\right),
 \end{align}
which should be interpreted in the obvious sense as $t$ varies in $\mathbb{R}$. Here $\mathrm{tn}(u,k)$ stands for the Jacobi elliptic function given by $\mathrm{sn}(u,k)/\mathrm{cn}(u,k)$. See \cite{elliptic_integrals} for more details on elliptic functions and useful identities. 
\bibliography{ref}
\bibliographystyle{plain}
\end{document}